\documentclass[11pt,a4paper]{article}

\usepackage{amsmath,amssymb,amsthm}
\usepackage[utf8]{inputenc}
\usepackage{xcolor}
\usepackage{pgf,tikz,color}
\usepackage{multirow}
\usepackage[framemethod=TikZ]{mdframed}
\usepackage{paralist}

\usepackage{lineno}

\numberwithin{equation}{section}

\usepackage{libertine}

\definecolor{vuboranje}{cmyk}{0,.78,1.,0}
\definecolor{vubbleu}{cmyk}{1,.8,.16,.03}

\usepackage[a4paper,margin=1.15in]{geometry}

\usepackage{hyperref}
\hypersetup{colorlinks = true, citecolor = vubbleu, linkcolor = vubbleu, urlcolor = vubbleu}
\usepackage{bookmark}

\usepackage[noabbrev]{cleveref}

\newtheorem{theorem}{Theorem}[section]
\newtheorem{prop}[theorem]{Proposition}
\newtheorem{cor}[theorem]{Corollary}
\Crefname{cor}{Corollary}{Corollaries}
\newtheorem{lemma}[theorem]{Lemma}

\newtheorem*{notation}{Notation}

\theoremstyle{definition}
\newtheorem{defin}[theorem]{Definition}
\newtheorem{remark}[theorem]{Remark}
\newtheorem{exa}[theorem]{Example}

\Crefname{algo}{Algorithm}{Algorithms}

\def\cA{\mathcal A}

\def\cC{\mathcal C}

\def\cE{\mathcal E}

\def\cP{\mathcal P}

\def\Z{\mathbb Z}

\def\C{\mathbb C}

\def\PG{\mathrm{PG}}

\def\Fq{\mathbb{F}_q}

\DeclareRobustCommand{\ZpZ}[1]{\ifthenelse{\equal{#1}{1}}{\Z/p\Z}{{(\Z/p\Z)^{#1}}}}

\def\PGL{\mathrm{PGL}}

\def\PSp{\mathrm{PSp}}
\def\PGO+{\mathrm{PGO}^+}
\def\PGO-{\mathrm{PGO}^-}
\def\PGO{\mathrm{PGO}}

\def\PGU{\mathrm{PGU}}

\def\Sym{\mathrm{Sym}}
\DeclareMathOperator{\rk}{rk}
\newcommand{\erz}[1]{\langle #1 \rangle}
\newcommand{\gauss}[2]{{#1\brack #2}_q}
\newcommand{\PS}{\mathrm{PS}}

\DeclareMathOperator{\rank}{rank}
\DeclareMathOperator{\col}{col}
\DeclareMathOperator{\opp}{opp}

\title{An algebraic approach to Erd\H{o}s-Ko-Rado sets of flags in spherical buildings II}
\author{Jan De Beule, Sam Mattheus, Klaus Metsch}
\date{}

\begin{document}

\maketitle

\begin{abstract}
	We continue our investigation of Erd\H{o}s-Ko-Rado (EKR) sets of flags in spherical buildings. In previous work, we used the theory of buildings and Iwahori-Hecke algebras to obtain upper bounds on their size. As the next step towards the classification of the maximal EKR-sets, we describe the eigenspaces for the smallest eigenvalue of the opposition graphs. We determine their multiplicity and provide a combinatorial description of spanning sets of these subspaces, from which a complete description of the maximal Erd\H{o}s-Ko-Rado sets of flags may potentially be found. This was recently shown to be possible for type $A_n$, $n$ odd, by Heering, Lansdown, and the last author by making use of the current work. \\[1em]
\end{abstract}

\vspace{-1em}

\renewcommand{\thefootnote}{\fnsymbol{footnote}} 
\footnotetext{\noindent
	\textbf{Address of the authors}\\
	\{jan.de.beule,sam.mattheus\}@vub.be\\
	Department of Mathematics,
	Vrije Universiteit Brussel,\\
	Pleinlaan 2,
	B-1050 Elsene,
	Belgium\\
	The second author is supported by a postdoctoral fellowship 1267923N from the Research Foundation Flanders (FWO).

	\noindent
	klaus.metsch@math.uni-giessen.de\\
	Justus-Liebig-Universität, 
	Mathematisches Institut,\\ 
	Arndtstra\ss e 2, 
	D-35392, Gie\ss en, 
	Germany}     
\renewcommand{\thefootnote}{\arabic{footnote}} 

\tableofcontents 

\clearpage

\section{Introduction}

In previous work \cite{DBMM22}, we proved upper bounds for the size of Erd\H{o}s-Ko-Rado (EKR) sets in finite spherical buildings. These correspond to independent sets in the so-called opposition graph and bounds on their size were obtained by techniques from spectral graph theory, in particular the ratio bound. When considering EKR problems, the next step after obtaining bounds is to investigate sharpness and classify the extremal examples. The current work is a first step in that direction, and to the best of our knowledge, the first instance of progress towards classification where the underlying adjacency algebra is non-commutative. As such, known methods in the literature such as the theory of polytopes, dual width and spectral techniques do not seem to generalize easily to this setting, see the introduction of \cite{FL22} and the book \cite{GM16} for an extensive overview of these techniques. What we can do is use the property, which follows from equality in the ratio bound, that the characteristic vectors of maximal EKR-sets must be contained in the sum of the eigenspaces of the largest and smallest eigenvalues. This technique is sometimes referred to as the rank method and was pioneered by Godsil and Newman \cite{Newman04} and further developed by Godsil and Meagher \cite{GM16}. \\

The first step in this approach is to closely investigate these eigenspaces, by for example constructing a combinatorially defined basis of these spaces. One can then leverage this structural information to deduce a classification result, see \cite{GM16,Newman04}. Our contribution here is that we will construct combinatorially defined spanning sets of the relevant eigenspaces. It turns out that there are notable differences with classical EKR problems (such as sets and vector spaces). For instance, in the classical problems the sharp examples are so-called trivially intersecting families. These families span the relevant eigenspaces and the problem reduces to showing that there are no other characteristic functions of maximal EKR-sets contained in these eigenspaces. This is no longer the case for us. The dimensions of the eigenspaces are much larger (see \Cref{thm:multonmodules}) and the conjectural families of maximal EKR-sets span only a fraction of these spaces. This considerably increases the difficulty in obtaining classification results. However, using the results from this paper, Heering, Lansdown and the third author \cite{HLM24} have recently been able to classify maximal EKR-sets of flags in projective spaces of odd dimension, thereby confirming a conjecture from our earlier paper \cite{DBMM22}. \\

We will use the same set-up as in our previous work \cite{DBMM22}, and refer to it for all terminology and concepts that are not explained here. In particular, we will investigate the Iwahori-Hecke algebras associated to the spherical buildings, and their representations. A standard reference to this topic is the excellent book by Geck and Pfeiffer \cite{GP00}. Even though the structure of these algebras is quite well-understood over $\mathbb{C}$, the coordinatization in terms of maximal flags in finite spherical buildings seems to not have been considered in the literature before. For instance, the irreducible modules of the Iwahori-Hecke algebras are well-understood and can be described using standard tableaux \cite[Section 10.1]{GP00}. However, this does not provide us with any information since we work with a basis indexed by maximal flags, and it is unclear if one can describe the appropriate change-of-basis matrix. As such, we will build spanning sets of the relevant irreducible modules from scratch. \\

We will find these eigenvectors by prescribing invariance under an automorphism group of the opposition graph. This reduces the problem from finding eigenvectors of the large adjacency matrix of the opposition graph, to computations with a much smaller quotient matrix. In this way, we will be able to do the necessary computations by hand in general. In order to then show that the obtained eigenvectors span the eigenspace, we will relate the opposition graph on maximal flags to the opposition graph on partial flags. It is interesting that the eigenspaces of the smaller graph essentially control those of the larger, and it remains to be seen if this interaction could provide another approach to the classification problem. \\

We remark here that although our method is rather general, we will not treat all types of buildings in the same depth. To be precise: in type $A_n$, we will focus on odd $n$ for two reasons. Firstly, the bounds obtained in \cite{DBMM22} contain the expression $q^{(n+1)/2}$, which in general is not an integer for even $n$. Secondly, we don't have a good conjecture about what the maximal EKR-sets should be, but it may be the case that their description is also not too involved, as shown recently for $n = 4$ by Heering \cite{Heering24}.

In type $B_n$, there are up to two irreducible representations in which the smallest eigenvalue appears (see \Cref{thm:multonmodules}). Out of these, we will mainly concern ourselves with the reflection representation indexed by $([n-1],[1])$ \cite[Section 8.1.9]{GP00}. The reason for this is that the other irreducible representation indexed by $(\emptyset,[n])$ seems to either give rise to somewhat unexpected maximal EKR-sets when $G \in \{\PSp(2n,q),\PGO(2n+1,q)\}$ or we don't even have a conjecture for the size or structure of maximal EKR-sets when $G = \PGU(4n+2,q^2)$. This phenomenon already appears when we consider the opposition graph on generators \cite{PSV11}, and it seems unlikely that the rank method will be able to prove classification results in these cases. 

Finally, when the building is of type $D_n$, the opposition graph is intimately related to the opposition graph in type $B_n$ with structural parameter $e = 0$ (\Cref{lem:typeDfromtypeB}). As such, questions about the opposition graph in type $D_n$ can be easily transferred to type $B_n$ and studied there along with the other buildings of type $B_n$. \\

\noindent\textbf{Results.} By investigating the non-commutative adjacency algebras, we are able to make progress on the question of classifying maximal EKR-sets of flags in finite spherical buildings. Our main results are the computation of the multiplicity of the smallest eigenvalue in all types (\Cref{thm:multonmodules}) and the description of combinatorially defined spanning sets for all types (see \Cref{sec:specifictypes} for the details). The former relies heavily on understanding the Iwahori-Hecke algebras and interpreting classical results in our context. The latter is proven using a blend of ideas involving group actions, quotient matrices and linear algebra. \\

For the sake of the reader, we give the description of the eigenvectors below for type $A_n$ and in a specific module for type $B_n$. Recall that the associated geometries are vector (or equivalently, projective) spaces and polar spaces defined over finite fields.

\begin{theorem}
	Denote by $\cC$ the set of maximal flags in $\Fq^{2n}$, i.e.
	\[\cC = \{(U_0,U_1,\dots,U_{2n}) \,\, | \,\, \{0\} = U_0 \subset U_1 \subset \dots \subset U_{2n-1} \subset U_{2n} = \Fq^{2n}, \, \dim(U_i) = i\},\]
	and let $P$ be a one-dimensional subspace of $\Fq^{2n}$. For $j \in [n]$ define the function $\chi_j^P:\mathbb{C}^\cC \to \mathbb{C}$ by
	\begin{align*}
	\chi_j^P(U_0,U_1,\dots,U_{2n}) = \begin{cases}
	0 &\text{if } P \subseteq U_{n-j}\\
	q^j &\text{if } P \subseteq U_n \setminus U_{n-j} \\
	-1 &\text{if } P \subseteq U_{n+j} \setminus U_n\\ 
	0 &\text{if } P \notin U_{n+j}.
	\end{cases}
	\end{align*}
	
	Then the set of functions $\{\chi_j^P\}$, where $j \in [n]$ and $P$ ranges over all one-dimensional subspaces of $\Fq^{2n}$, is a spanning set for the eigenspace of the smallest eigenvalue of the opposition graph on maximal flags.
\end{theorem}

\begin{remark}
	We will see later (cf.\ \Cref{thm:multonmodules}) that the multiplicity of the smallest eigenvalue in this case equals $n\frac{q^{2n}-q}{q-1}$ . Moreover, from our computations it will follow that we can extract a basis from this spanning set by considering the set $\{\chi_j^P\}$, where $j \in [n]$ and $P$ ranges over all \textit{but one} one-dimensional subspaces of $\Fq^{2n}$.
\end{remark}

The description in type $B_n$ is more delicate as the eigenspace of the smallest eigenvalue could appear on different irreducible modules of the corresponding Iwahori-Hecke algebras. Nevertheless, the part of the eigenspace that is contained in the `nice' module $M_{[n-1],[1]}$ can be described in a similar way as in type $A_n$. We will denote a polar space of rank $n$, parameter $e$ and defined over $\Fq$ by $\PS(n,e,q)$ (see also the next section) and comes equipped with a polarity, which we denote by $\perp$. We denote by $\rk(U)$ the rank or vector dimension of a subspace $U$ contained in the polar space.

\begin{theorem}
	Denote by $\cC$ the set of maximal flags in $\PS(n,e,q)$, i.e.
	\[\cC = \{(U_0,U_1,\dots,U_{n}) \,\, | \,\, \emptyset = U_0 \subset U_1 \subset \dots \subset U_n, \, \rk(U_i) = i\},\]
	and let $P$ be a point in $\PS(n,e,q)$. For $j \in [m]$ define the function $\chi_j^P:\mathbb{C}^\cC \to \mathbb{C}$ by
	\begin{align*}
	\chi_j^P(U_1,\dots,U_{2n}) = \begin{cases}
	0 &\text{if } P \subseteq U_{n-j}\\
	q^{j+e-1} &\text{if } P \subseteq U_n \setminus U_{n-j} \\
	-1 &\text{if } P \subseteq U_{n-j}^\perp \setminus U_n\\ 
	0 &\text{if } P \notin U_{n-j}^\perp.
	\end{cases}
	\end{align*}
	
	Then the set of functions $\{\chi_j^P\}$, where $j \in [m]$ and $P$ ranges over points, is a spanning set for the part of the eigenspace of the smallest eigenvalue of the opposition graph on maximal flags that is contained in the module $M_{([n-1],[1])}$.
\end{theorem}

\noindent\textbf{Structure of the paper.} First we will recall the set-up and the results of \cite{DBMM22} in \Cref{sec:previous} and state which irreducible representations we will investigate closer. Our strategy will be to first determine the multiplicities of the relevant eigenspace in each irreducible module (\Cref{sec:multiplicity}) and then find a spanning set for these spaces. The general method for the latter is outlined in \Cref{sec:spanningstrategy} and applied to the different types in \Cref{sec:specifictypes}. \\

\noindent\textbf{Acknowledgment.} We would like to thank James Parkinson for his help in understanding the connections between the theory of buildings, Iwahori-Hecke algebras and geometries over finite fields, both in this paper and the previous.

\section{Previous results} \label{sec:previous}

In \cite{DBMM22} we started from classical groups of Lie type and their corresponding buildings, as pioneered by Tits \cite{Tits74}. These groups are listed in \Cref{table:classicalgroups}, along with their corresponding geometries and Weyl groups. These Weyl groups are an indexing set for a basis of the Iwahori-Hecke algebras $\cA(G,G/B)$ corresponding to each group $G$, which are the centralizer algebras of $\mathbb{C}[G/B]$, where $B$ is the standard Borel subgroup \cite[Section 8.4]{GP00}. \\

\begin{table}[h!]
	\centering
	\setlength{\tabcolsep}{5mm}
	\def\arraystretch{1.4} 
	\begin{tabular}{l | l | l | l }
		classical group & geometry & Cartan notation & Weyl group \\
		\hline
		$\PGL(n+1,q)$ & projective space & $A_n(q)$ & $A_n$ \\
		$\PGU(2n+1,q^2)$ & hermitian polar space & $^2A_{2n}(q^2)$ & $B_n$ \\
		$\PGU(2n,q^2)$ & hermitian polar space & $^2A_{2n-1}(q^2)$ & $B_n$ \\
		$\PGO(2n+1,q)$ & parabolic quadric & $B_n(q)$ &  $B_n$ \\
		$\PSp(2n,q)$ & symplectic polar space & $C_n(q)$ & $B_n$ \\
		$\PGO^-(2n+2,q)$ & elliptic quadric & $^2D_{n+1}(q^2)$ & $B_{n}$ \\
		$\PGO^+(2n,q)$ & hyperbolic quadric & $D_n(q)$ & $D_n$ 		
	\end{tabular}
	\caption{The projective classical groups.}
	\label{table:classicalgroups}		
\end{table}

The \textbf{opposition graph} of the building can then be defined as the graph with adjacency matrix $A_{w_0}$, where $w_0$ is the longest word of the Weyl group $W$. We record the infinite families of Weyl groups, their Dynkin diagrams and the length of the longest word in \Cref{table:weylgroups} for future reference. In a more geometric way, the opposition graph is the graph whose vertex set are the maximal flags in a spherical building, adjacent whenever they are opposite in a building-theoretical sense. \\


\begin{table}[h!]
	\centering
	\setlength{\tabcolsep}{8mm}
	\def\arraystretch{4} 
	\begin{tabular}{c | c | c}
		Weyl group & Dynkin diagram & $\ell(w_0)$ \\
		\hline
		$A_n$ & \begin{tikzpicture}[very thick, baseline={(N1)},scale=0.8]
		\def\a{1}
		\tikzset{dynkin/.style={circle,draw,minimum size=0.5mm,fill}}
		\path
		(0,0)      node[dynkin] (N1) {} 
		+(90:.5) node{$1$}
		
		++(0:2*\a) node[dynkin] (N2) {} 
		+(90:.5) node{$2$}
		
		++(0:\a)   coordinate (A) ++(0:\a) coordinate (B)
		++(0:\a) node[dynkin] (N3) {}
		+(90:.5) node{${n-1}$}
		
		++(0:2*\a) node[dynkin] (N4) {} 
		+(90:.5) node{$n$};
		
		\draw[dashed] (A)--(B);
		\draw (N1)--(N2)--(A) (B)--(N3)--(N4);
		\end{tikzpicture} & $\dfrac{n(n+1)}{2}$ \\
		
		$B_n$ & \begin{tikzpicture}[very thick,baseline={(N1)},scale=0.8]
		\def\a{1}
		\tikzset{dynkin/.style={circle,draw,minimum size=0.5mm,fill}}
		\path
		(0,0)      node[dynkin] (N1) {} 
		+(90:.5) node{$1$}
		
		++(0:2*\a) node[dynkin] (N2) {} 
		+(90:.5) node{$2$}
		
		++(0:\a)   coordinate (A) ++(0:\a) coordinate (B)
		++(0:\a) node[dynkin] (N3) {}
		+(90:.5) node{${n-1}$}
		
		++(0:2*\a) node[dynkin] (N4) {} 
		+(90:.5) node{$n$};
		
		\draw[dashed] (A)--(B);
		\draw (N1)--(N2)--(A) (B)--(N3);
		\draw[double distance=2pt] (N3)--(N4);
		\end{tikzpicture} & $n^2$ \\
		
		$D_n$ & \begin{tikzpicture}[very thick,baseline={(N1)},scale=0.8]
		\def\a{1}
		\tikzset{dynkin/.style={circle,draw,minimum size=0.5mm,fill}}
		\path
		(0,0)    node[dynkin] (N1) {} 
		+(90:.5) node{$1$}
		
		++(0:2*\a) node[dynkin] (N2) {} 
		+(90:.5) node{$2$}
		
		++(0:\a) coordinate (A) 
		++(0:\a) coordinate (B)
		++(0:\a) node[dynkin] (N3) {}
		+(90:.5) node{$n-2$}
		
		++(20:1.7*\a) node[dynkin] (N4) {} 
		+(90:.5) node{$n-1$}
		
		(N3) ++(-20:1.7*\a) node[dynkin] (N5) {} 
		+(90:.5) node{$n$};
		
		\draw[dashed] (A)--(B);
		\draw (N1)--(N2)--(A) (B)--(N3)--(N4) (N3)--(N5);
		\end{tikzpicture} & $n(n-1)$ \\
	\end{tabular}
	\caption{The Weyl groups $A_n$, $B_n$ and $D_n$.}
	\label{table:weylgroups}
\end{table}

The Iwahori-Hecke algebra corresponding to polar spaces depends on an extra parameter $e$. Its geometrical significance lies in the fact that there are $q^e+1$ rank $n$ spaces through an arbitrary rank $n-1$ space. One can then deduce that this parameter equals $3/2$, $1/2$, $1$, $1$, $2$ and $0$ respectively for the polar spaces listed in \Cref{table:classicalgroups}. As such, we will denote a polar space of rank $n$ with parameter $e$ and defined over $\Fq$ by $\PS(n,e,q)$. \\


One can study the spectrum of $A_{w_0}$ by looking at the Iwahori-Hecke algebra as a whole, and deduce its eigenvalues on each module. In this way, we were able to compute the largest and smallest eigenvalue of $A_{w_0}$ which we could then feed into the ratio bound for independent sets \cite{DBMM22}. We can summarize the results obtained there as follows. Recall that irreducible modules in type $A_n$ are indexed by partitions of $n$, whereas in type $B_n$ they are indexed by pairs of partitions, whose total size is $n$.

\begin{theorem}\label{maintheoremprev}
	The smallest eigenvalue $\lambda_{\min}$ of $A_{w_0}$ is known in all cases and its value and the modules on which it appears are indexed by are as follows:
	
	\begin{itemize}
		\item $\lambda_{\min} = -q^{(n^2-1)/2}$ on the module indexed by $[n-1,1]$ for type $A_n$,
		\item $\lambda_{\min} = -q^{(n-1)(n+e-1)}$ on the module indexed by $([n-1],[1])$ for type $B_n$, $n$ even or $e \geq 1$,
		\item $\lambda_{\min} = -q^{n(n-1)}$ on the module indexed by $(\emptyset,[n])$ for type $B_n$, $n$ odd and $e \leq 1$,
		\item $\lambda_{\min} = -q^{n(n-1)}$ on the module indexed by $([1],[n-1])$ for type $B_n$, $n$ even and $e = 0$,
		\item $\lambda_{\min} = -q^{(n-1)^2}$ on the module indexed by $([n-1],[1])$ for type $D_n$.
	\end{itemize}
\end{theorem}

\begin{remark}
	Remark that in type $B_n$ for $e = 1$ and $n$ odd, the smallest eigenvalue appears on two different irreducible submodules. This peculiarity for odd rank can also be seen geometrically in the classification of EKR-sets of generators of $Q(2n,q)$ and $W(2n-1,q)$, $q$ even, in \cite{PSV11}. In this case, there are two distinct families of maximal EKR-sets of generators: the trivially intersecting families and one family of generators of an embedded hyperbolic quadric $Q^+(2n-1,q)$.
\end{remark}

\begin{remark}
	For type $B_n$, $e = 0$ and $n$ even, the opposition graph is the disjoint union of two isomorphic graphs, which explains why the module $([1],[n-1])$ appears. Similarly, $([n],\emptyset)$ and $(\emptyset,[n])$ each give rise to the maximal eigenvalue. For $e = 0$ it thus makes sense to focus on only one of the components. This is essentially what happens by looking at the hyperbolic quadric as a building of type $D_n$ (i.e.\ the oriflamme geometry, see \cite[Section 4.5.2]{BVM22}) instead of type $B_n$ with $e=0$.
	
	On the other hand, when $n$ is odd, we will see later on that the opposition graph in type $B_n$, $e = 0$, is the bipartite double of that of type $D_n$. We will therefore be able to transfer spectral information from one to the other.
	
	Recall that on an algebraic level, a similar `halving' from type $B_n$ to type $D_n$ occurs on the level of representations of the corresponding Weyl groups: each pair of representations corresponding to $(\mu,\nu)$ and $(\nu,\mu)$ in type $B_n$ restricts to the same representation in type $D_n$ when $\mu \neq \nu$ \cite[Section 10.4]{GP00}.
\end{remark}


\section{Multiplicity of the smallest eigenvalue}\label{sec:multiplicity}

First we will derive the dimension of the eigenspace of the smallest eigenvalue. This computation mostly builds on foundations that has been established before, and most of the work lies in translating the relevant concepts from the theory of Iwahori-Hecke algebras. We again refer to \cite{GP00} for a reference of the following facts.

The Iwahori-Hecke algebra $\cA(G,G/B)$ (denoted as $\mathrm{End}_G(\mathbb{C}[G/B])$ in \cite[Section 8.4]{GP00}) is a complex semi-simple algebra. This means that if we consider the action on $G/B$, the corresponding permutation representation $\rho$ splits into irreducible representations as
\begin{align}
	\rho = \bigoplus_{i \in I}m_i \rho_i, \label{eq:decomposition}
\end{align}

where $\{\rho_i\}_{i \in I}$ are the irreducible representations of $\cA(G,G/B)$ and $m_i$ is the corresponding multiplicity. The indexing set $I$ depends on the type of the underlying Weyl group:

\begin{itemize}
	
	\item type $A_n$: $I$ is the set of partitions of $[n+1]$, 
	\item type $B_n$: $I$ is pairs of partitions $(\mu,\nu)$ such that $|\mu| + |\nu| = n$,
	\item type $D_n$: representations are restrictions of those of $B_n$. Either $\mu \neq \nu$ and the restrictions of $\rho_{(\mu,\nu)}$ and $\rho_{(\nu,\mu)}$ coincide or the restriction of $\rho(\mu,\mu)$ splits into irreducibles as $\rho_{(\mu,+)}+\rho_{(\mu,-)}$.
\end{itemize}	 

Viewing the elements of $\cA(G,G/B)$ not as abstract algebra elements, but as square matrices of size $|G/B|$ over $\C$ -- which one usually does in the theory of association schemes -- the previous boils down to block-diagonalizing the algebra. From this point of view, we can rewrite \eqref{eq:decomposition} as an isomorphism of matrix algebras over $\C$:
\begin{align}
	\cA(G,G/B) \cong \bigoplus_{i \in I}M_{n_i}(\C) \otimes I_{m_i}. \label{eq:matrixdecomposition}
\end{align}

In this case, the representation $\rho_i$ in \eqref{eq:decomposition} is the projection onto the block $M_{n_i}(\C)$ and hence $n_i = \rho_i(1)$. \\

The largest eigenvalue is only attained on the trivial representation and has multiplicity one. This also follows from the Perron-Frobenius theorem in algebraic graph theory. The smallest eigenvalue $\lambda_{\min}$ on the other hand could be attained by several representations as we see in \Cref{maintheoremprev}. In order to determine its multiplicity, a quick look at \eqref{eq:matrixdecomposition} reveals that we can do the following:
\begin{enumerate}
	\item compute the multiplicity $\rm mult_i$ of $\lambda_{\min}$ on each of the blocks $\rho_i(A_{w_0})$,
	\item find the value $m_i$, 
	\item compute the total multiplicity as $\sum_{i \in I} {\rm mult_i}\cdot m_i$.
\end{enumerate}

For both of the first two points, we can make extensive use of existing literature: the representations on which $\lambda_{\min}$ appears, are listed in \Cref{maintheoremprev}. There are at most two representations per type and all of them have been studied before. The value $m_i$ is known in the literature as the \textit{generic degree} and is known for the classical types \cite[Section 10.5]{GP00}.

\begin{theorem}\label{thm:multonmodules}
	Let $\lambda_i$ be the smallest eigenvalue attained by $A_{w_0}$ on the irreducible module $M$ indexed by $i$ and ${\rm mult}_i$ its multiplicity on $M$. For each module $M$ of interest mentioned in \Cref{maintheoremprev}, we list its index, smallest eigenvalue $\lambda_i$, its multiplicity ${\rm mult}_i$, and the generic degree of $M$:
	
	\begin{table}[h!]
		\setlength{\tabcolsep}{3mm}
		\def\arraystretch{2.5} 
		\begin{center}
		\begin{tabular}{c | c | c | c | c}
		type & index $i$ &$ \lambda_i$ &  ${\rm mult}_i$ & generic degree \\ \hline
		$A_n$ & $[n,1]$ & $-q^{(n^2-1)/2}$ & $\lceil n/2 \rceil$ & $\dfrac{q^{n+1}-q}{q-1}$ \\
		$B_n$ & $([n-1],[1])$ & $-q^{(n-1)(n+e-1)}$ & $n$ & $\dfrac{q^e(q^n-1)(q^{n+e-2}+1)}{(q-1)(q^{e-1}+1)}$ \\
		$B_n$ & $(\emptyset,[n])$ & $(-1)^nq^{n(n-1)}$ & 1 & $q^e\dfrac{(q^{n+e-1}+1)(q^{n+e-2}+1)\cdots(q^{e+1}+1)}{(q^{n-e-1}+1)(q^{n-e-2}+1)\cdots(q^{-e+1}+1)}$ \\
		$D_n$ & $([n-1],[1])$ & $-q^{(n-1)^2}$ & $n$ & $\dfrac{(q^{n+1}-q)(q^{n-2}+1)}{(q-1)(q+1)}$
	\end{tabular}
	\end{center}
	\end{table}
\end{theorem}

\begin{proof}
	Explicit expressions are known for $\rho(A_{w_0})$ for each of the representations $\rho$ whose corresponding module is listed above. For the $n$-dimensional reflection representation of a Weyl group, which is the case for all but the third row, this can be found in Kilmoyer's thesis \cite[Proposition 28]{Kilmoyerthesis}. From this result, we can deduce the following explicit expressions for the image of $A_{w_0}$ under the reflection representation: \\
	
	\textbf{\fbox{type $A_n$}} \,\,
	The representation $\rho_{[n,1]}$ is $n$-dimensional and we have
	\[\rho_{[n,1]}(A_{w_0}) = \begin{pmatrix}
		0  & \cdots & 0 & -q^{(n^2-1)/2} \\
		0  & \cdots & -q^{(n^2-1)/2} & 0 \\
		\vdots  & \reflectbox{$\ddots$} & \vdots & \vdots \\
		-q^{(n^2-1)/2} & 0 & \cdots & 0 
	\end{pmatrix}\]
	So ${\rm mult}_{[n,1]}$ is $n/2$ or $(n+1)/2$, depending on whether $n$ is even or odd respectively. \\
	
	\textbf{\fbox{type $B_n$}} \,\, $\rho_{([n-1],[1])}(A_{w_0}) = -q^{(n-1)(n+e-1)}I_{n}$ and hence ${{\rm mult}_{([n-1],[1])}} = n$. The representation $\rho_{(\emptyset,[n])}$ is one-dimensional (this follows for example from \cite[Theorem 10.1.5]{GP00}) and thus ${\rm mult}_{(\emptyset,[n])} = 1$. \\
	
	\fbox{\textbf{type $D_n$}} \,\, $\rho_{([n-1],[1])}(A_{w_0}) = -q^{(n-1)^2}I_n$ and hence ${{\rm mult}_{([n-1],[1])}} = n$. \\
	
	The last column can be immediately deduced from existing results. The generic degrees of the reflection representation in each type can be found in \cite[Example 10.5.8]{GP00}. For the module $M_{(\emptyset,[n])}$ in type $B_n$, we can use \cite[Theorem 10.5.3]{GP00} or refer to \cite[values for $d_{\sigma_1}$ on p114-115]{CIK71} for simplified expressions for each value of $e$ as given in \Cref{table:mults}. 
\end{proof}

	One can now combine \Cref{maintheoremprev} and \Cref{thm:multonmodules} as indicated before to obtain the multiplicity of $\lambda_{\min}$ for each of the types. Since there are eight cases to consider, we will not list explicit multiplicities here, but refer to the detailed \Cref{table:mults} in the appendix.
	
	\begin{remark}
		From the values of the generic degrees, one can already suspect a relation to graphs defined on partial flags. To be precise: the graph on the points of $\PG(n,q)$, two points adjacent whenever they are distinct, is the complete graph and hence the multiplicity of the smallest eigenvalue is exactly the generic degree listed in the first row. In type $B_n$, the opposition graph on points of a polar space $\PS(n,e,q)$ has the generic degree listed in the second row as the multiplicity for its smallest eigenvalue, see \cite[Theorem 2.2.12]{BVM22}. The multiplicity of the smallest eigenvalue of the opposition graph on generators of $\PS(n,e,q)$ equals the generic degree listed in the third row, see \cite[Theorem 9.4.3]{BCN89}.
	\end{remark}

\section{A spanning set for the eigenspace} \label{sec:spanningstrategy}

We will outline the general idea to construct eigenvectors in a few steps and execute the detailed computations for each type separately later on. \\ 

\noindent \textbf{Notation and terminology.} We will use the notation $[n]:=\{1,\dots,n\}$ throughout. It will be clear from context whether $[n]$ denotes a set, or a partition. For example $i \in [n]$ will always refer to the former as we do not consider elements of a partition. The identity matrix and the all-one matrix indexed by a set $\Omega$ will be denoted by $I_\Omega$ and $J_\Omega$ respectively.

The set of maximal flags in a spherical building will be denoted by $\cC$. For $c \in \cC$ we will denote by $\opp(c)$ the flags that are opposite to it. Equivalently, they are the neighbours of $c$ in the opposition graph.
We will at times alternate between viewing $A_{w_0}$ as a matrix whose rows and columns are indexed by $\cC$ on the one hand, and as an operator on the space of functions $L^2(G/B)$ on the other hand, recalling that $\cC$ is in bijection with $G/B$. Consequently, we would in such instances talk about `eigenvectors', when in fact we are describing eigenfunctions. 
 \\

Given a parabolic subgroup $P_J$ of $G$, for some choice of $J$, we will construct eigenvectors left invariant by this group. In other words, the eigenvectors will be constant on orbits of $P_J$ on $G/B$. The specific choice of $J$ will be motivated both by making the computations as simple as possible and by the automorphism groups of known maximal EKR-sets of flags given in \cite{DBMM22}. 

The number of orbits of $P_J$ on the set of maximal flags equals $[W:W_J]$ as every (right) coset representative $w$ leads to a relation $P_JwB$ between the element fixed by $P_J$ and maximal flags. In order to have a compact description of the eigenvectors, we would therefore like to minimize $[W:W_J]$ for a given Weyl group $W$. For example, one can check that in type $A_n$, choosing $J = [n] \setminus \{k\}$ we find $[W:W_J] = {n+1 \choose k}$ and so $J = [n] \setminus \{1\}$ or dually $J = [n] \setminus \{n\}$ will give the least number of orbits. This corresponds geometrically to considering eigenvectors invariant under the stabilizer of a fixed point or hyperplane respectively. Both of these also happen to be the automorphism group of a maximal EKR-set of maximal flags in the case of type $A_n$, $n$ odd \cite{DBMM22}.

Furthermore, the choice of $J$ (in any type) with $|J|=n-1$ has the added benefit that the association scheme on $G/P_J$ is always commutative in type $A_n$ and $B_n$ and for our choice of $J$ in type $D_n$. We refer the interested reader to \cite{APVM13} for a complete study of when this algebra is commutative. The commutativity will allow us to show that the obtained set of vectors is spanning the eigenspace of $\lambda_{\min}$ in a neat way.

\subsection{Step 1: describe $P_J$-orbits on $G/B$} 

So fix a choice of $J$ and denote the elements $G/P_J$ by $\cE$. We will call them elements, since we will only deal with the case $|J|=n-1$ and hence $\cE$ are really elements of the incidence geometries and not partial flags. For example if $J = [n] \setminus \{1\}$, then $\cE$ is the set of points of the incidence geometry.

As discussed before, if we denote by $\{w_1,\dots,w_\ell\}$ the set of right coset representatives of $W_J$ in $W$, then every maximal flag will be in an orbit $P_Jw_iB$ for some $i \in [\ell]$. 

\begin{defin} \label{def:typegeneral}
	Let $X \in \cE$. The \textbf{type} of a maximal flag $gB$\, w.r.t.\ $X$ is $i$ if and only if $P_JgB = P_Jw_iB$.	
\end{defin}


\begin{notation}
	For any $i \in [\ell]$ we introduce the following notation:
	\begin{itemize}
		\item $\cC_i^X$ is the set of maximal flags which are of type $i$ w.r.t.\ $X$;
		\item $T_i$ is the $|\cC|\times|\cE|$ adjacency matrix of the `type $i$'-relation. That is, $T_i$ has rows and columns indexed by maximal flags and elements respectively, s.t. $T_i(c,X) = 1$ if and only if $c \in \cC_i^X$ and $0$ otherwise.
	\end{itemize}
\end{notation}

Remark that if we choose the coset representatives to have minimum length, then 
\begin{align}\label{eq:sizeofCiX}
	|\cC_i^X| = q^{\ell(w_i)}\sum_{w \in W_J}q^{\ell(w)},
\end{align}
as $|BwB/B| = q^{\ell(w)}$ in general. \\

\begin{exa}\label{exa:A3example}
	Consider the spherical building of type $A_3$, i.e.\ the geometry of maximal flags in $\PG(3,q)$, with associated Weyl group $W = \Sym(4)$. This group is generated by $\{s_1,s_2,s_3\}$, where $s_i = (i,i+1)$ are the usual adjacent transpositions. If we choose $J = \{2,3\}$ then $W_J \cong \Sym(3)$ and hence $[W:W_J] = 4$. The right coset representatives of $W_J$ in $W$ are $\{e,s_1,s_1s_2,s_1s_2s_3\}$, where $e$ denotes the identity element. The cosets $gP_J$ correspond to points in $\PG(3,q)$ and the type of a flag $hB$ depends on $P_Jg^{-1}hB$. Geometrically these four possibilities depend on the mutual position of the point $X = gP_J$ and the point-line-plane flag $(Y,\ell,\pi) = hB$ in $\PG(3,q)$:
	\begin{center}
		{\renewcommand{\arraystretch}{1.2}
		\begin{tabular}{l | l}
			geometry & $P_Jg^{-1}hB$ \\	\hline		
			$X = Y$ & $P_JeB = P_J$ \\
			$X \in \ell \setminus \{Y\}$ & $P_Js_1B$ \\
			$X \in \pi\setminus \ell$ & $P_Js_1s_2B$ \\
			$X \notin \pi$ & $P_Js_1s_2s_3B$
	\end{tabular}}
	\end{center}
	This completes the description of the sets $\cC_i^X$, and we leave it up to the reader to verify geometrically that \eqref{eq:sizeofCiX} determines their size.
\end{exa}

\subsection{Step 2: define eigenvectors via the quotient matrix}\label{subsection:step2} 

We will now construct our eigenvectors by finding them as the column vectors of suitable linear combinations of $T_i$. So let $F_j = \sum_{i = 1}^\ell f_{ij}T_i$, then the columns of $F_j$ are eigenvectors of $A_{w_0}$ with eigenvalue $\lambda_{\min}$ if and only if
\begin{align}\label{eq:findingeigenvectorsgeneral}
	A_{w_0}F_j = \lambda_{\min}F_j &\Leftrightarrow \sum_{i = 1}^\ell f_{ij}A_{w_0}T_i = \lambda_{\min}\sum_{i = 1}^\ell f_{ij}T_i.
\end{align}

We must therefore compute $A_{w_0}T_i$ for all $i \in [\ell]$. Observe that the entry $(A_{w_0}T_i)(c,X)$ equals $|\cC_i^X \cap \opp(c)|$. Since $\cC_i^X$ is an orbit of $P_J$, this number will depend on the orbit of $c$, and not on $c$ itself. In other words, given $i,j \in [\ell]$, we need to compute $|\cC_j^X \cap \opp(c)|$ for a given $c \in \cC_i^X$. This is precisely the content of quotient matrices. 

First we introduce the following notation.

\begin{notation}
	Let $\Omega$ be a finite set. If $S \subset \Omega$, then $\textbf{1}_S$ denotes the characteristic (column) vector of $S$ in $\C^\Omega$. For example, $\textbf{1}_\Omega$ denotes the all-one vector. Typically, $\Omega$ will be clear from context.
\end{notation}

We will consider the actions of a maximal parabolic subgroup $P_J$ and its conjugates on the opposition graph and consider the quotient matrix under this action.

\begin{defin}\label{def:quotientmatrix}
	Let $\Gamma$ be a graph and denote the orbits of $H \leq \mathrm{Aut}(\Gamma)$ on $V(\Gamma)$ by $\{O_1,\dots,O_k\}$. The \textbf{characteristic matrix} $P$ is the $|V(\Gamma)|\times k$ matrix whose columns are $\{\mathbf{1}_{O_1},\dots,\mathbf{1}_{O_k}\}$. If $A$ denotes the adjacency matrix of $\Gamma$, then the \textbf{quotient matrix} $Q$ is defined by the equation $AP = PQ$.
\end{defin}

We remark that typically, quotient matrices are defined in terms of equitable partitions, which generalize orbit partitions as in \Cref{def:quotientmatrix} in a combinatorial way. We refer the interested reader to \cite[Section 9.3]{GR01}. \\

In our case, the characteristic matrix $P$ depends on the choice of $X \in \cE$, but the quotient matrix $Q$ is the same for all $X \in \cE$ as $G$ acts transitively on $G/P_J$. Moreover, $Q_{ij} = |\cC_j^X \cap \opp(c)|$ for arbitrary $c \in \cC_i^X$, as wanted. When we have to compute the quotient matrix, the equality 
\begin{align}\label{eq:quotientmatrixdoublecounting}
	|\cC_i^X|Q_{ij} = |\cC_j^X|Q_{ji},
\end{align}
obtained through double counting, can essentially cut the work in half.

Furthermore, we find from our previous discussion that
\begin{align*}
	A_{w_0}T_i = \sum_{k=1}^{\ell}Q_{ki}T_k
\end{align*}
and hence combining this with \eqref{eq:findingeigenvectorsgeneral} we obtain
\begin{align*}
	\sum_{i = 1}^\ell f_{ij}\sum_{k=1}^{\ell}Q_{ki}T_k = \lambda_{\min}\sum_{i = 1}^\ell f_{ij}T_i.
\end{align*}

Since the $\{T_i\}_{i \in [\ell]}$ are a set of $\{0,1\}$-matrices that sum to the all-one matrix, they are linearly independent and hence we conclude $F_j$ has eigenvectors for $\lambda_{\min}$ as its columns if and only if
\begin{align*}
	\sum_{i = 1}^\ell Q_{ki}f_{ij}=\lambda_{\min}f_{kj}.
\end{align*} 

In other words, if and only if $(f_{1j},\dots,f_{\ell j})^\top$ is a right eigenvector of $Q$. This is a condensed version of `lifting' eigenvectors from a quotient matrix as described in \cite[p197-198]{GR01}. The point is that by definition, the quotient matrix depends a priori on $P_J$ or equivalently, the choice of $X \in \cE$. However, as mentioned before, the resulting matrix $Q$ will be the same for all choices. The eigenvectors are hence the same, but they will lift to different vectors in $\C^{\cC}$. By defining the type matrices $\{T_i\}_{i \in [\ell]}$, we can essentially record these lifts all at once in $F_j$.

\subsection{Step 3: prove that the eigenvectors span the space} 

For each of the four modules listed in \Cref{thm:multonmodules}, we will make a specific choice of $J$. It turns out that in each case the multiplicity $m$ of $\lambda_{\min}$ as an eigenvalue of $Q$, will equal $\mathrm{mult}_i$. So given $m$ linearly independent eigenvectors of $Q$ and hence $m$ matrices $\{F_1,\dots,F_m\}$ defined as above, we will show that
\begin{enumerate}[(a)]
	\item $\rank(F_j)$ is at least the generic degree for all $i \in [m]$, and
	\item the dimension of $\col(F_1)+\dots+\col(F_m)$ is at least $m$ times the generic degree.
\end{enumerate} 
This shows that the column spaces span the part of the eigenspace contained in the relevant module. Unfortunately, except for type $A_n$, it does not seem feasible to extract an explicit basis from this set. \\

We will achieve both goals simultaneously. In particular, we will compute $T_i^\top F_j$ for $i \in [\ell], j \in [m]$. Since $F_j$ is a linear combination of the $T_i$, this boils down to computing $T_i^\top T_j$ for $i,j \in [\ell]$ for which we can observe that $(T_i^\top T_j)(X,Y) = |\cC_i^X \cap \cC_j^Y|$. Once again, this number will not depend on the specific pair $(X,Y)$ but on the orbit of $G$ on pairs of elements. In other words, the matrix $T_i^\top T_j$ will be an element of the association scheme on elements. \\

For example, in type $A_n$ for the choice $J = n \setminus \{1\}$, there are only two orbits ($X = Y$ and $X \neq Y$) and hence $T_i^\top T_j$ will be a linear combination of the $I_\cE$ and $J_\cE$. For our choices of $J$, the association scheme on elements $\cA(G,G/P_J)$ will be commutative and we can denote by $\{A_0=I_{\cE},A_1,\dots,A_d\}$ its basis $\{0,1\}$-matrices, where $d+1$ is the number of orbits of $G$ on pairs of elements. It follows that
\begin{align}
	T_i^\top T_j = \sum_{k=0}^{d}t^k_{ij}A_k, \label{eq:expandproductsofTinassocscheme}
\end{align}
where $t^k_{ij} = |\cC_i^X \cap \cC_j^Y|$ for any two elements $(X,Y)$ in relation $k$. \\

Now denote the basis of idempotents by $\{E_0 = \frac{1}{|\cE|}J,E_1,\dots,E_d\}$. These matrices satisfy $E_iE_j = \delta_{ij}E_i$ and $A_kE_r = p_k(r)E_r$, where $p_k(r)$ is the eigenvalue of $A_k$ on $E_r$. In the theory of commutative association schemes, these eigenvalues are collected in the $P$-matrix of the scheme, which is defined by $P_{r,k} = p_k(r)$. 

Recall that given a module $M$ from \Cref{thm:multonmodules}, we chose $J$ accordingly. The choice will be made such that the module is also present as an idempotent in the association scheme on elements. To be precise, in the decomposition of $\cA(G,G/P_J)$, similarly as in \eqref{eq:matrixdecomposition}, we find $M$ as an irreducible component. Since $\cA(G,G/P_J)$ is commutative, this is equivalent to saying that $M$ is isomorphic as a $\cA(G,G/P_J)$-module to a primitive idempotent, say $E_s$. \\

This can also be seen in a different way. Retaking the example of type $A_n$ and $J = [n] \setminus \{1\}$, we will see later that 
\[\mathrm{ind}_{P_J}^G(1_{P_J}) = \chi_{[n+1]} + \chi_{[n,1]}.\]
This implies that the block-diagonalization is
\[\cA(G,G/P_J) \cong \C \oplus \C I_{|\cE|-1}, \]
as the respective generic degrees are $1$ and $|\cE|-1 = \frac{q^n-q}{q-1}$. In this case, the module $M$ corresponds to the second term in the decomposition and equals the idempotent $E_1 = I_\cE - \frac{1}{|\cE|}J_\cE$. \\

We are now in the position to show our method for proving (a) and (b) simultaneously. Throughout, $E_s$ denotes the special idempotent isomorphic to $M$, whose rank equals the generic degree of $M$.

\begin{defin}\label{def:trianglecriterion}
	Given matrices $\{F_1,\dots,F_m\}$ as defined before, we say that $g:[m]\rightarrow[\ell]$ satisfies the \textbf{triangular criterion} if and only if
	
	\begin{align}\label{eq:triangularcriterion}
		T_h^\top F_j = \begin{cases}
			\alpha_jE_s & \text{if } h = g(j) \\
			0 & \text{if } h < g(j),
		\end{cases}
	\end{align}
	for some non-zero constants $\alpha_j$, $j \in [m]$.
\end{defin}

This criterion resembles a similar one in algebraic combinatorics from which it borrows the name. In that setting, it is typically used to prove the independence of a set of polynomials, see for example \cite[Proposition 2.5]{BF2022}.

\begin{lemma}\label{lem:triangularcriterion}
	If there exists $g:[m]\rightarrow[\ell]$ satisfying the triangular criterion then (a) and (b) hold.
\end{lemma}

\begin{proof}
	By definition we have $\rank(F_j) \geq \rank(T_{g(j)}F_j) = \rank(E_s)$. Since $\rank(E_s)$ is the generic degree of $M$, this proves (a).
	
	In order to prove (b), we observe that $T_{g(j)}^\top F_j = \alpha_jE_s$ implies that the image of the map $f_j:\col(F_j) \rightarrow \C^\cE, v \mapsto T_{g(j)}^\top v$ has dimension at least $\rank(E_s)$. We can then restrict $f_j$ to a subspace $S_j \leq \col(F_j)$ whose dimension equals $\rank(E_s)$ such that $f_j|_{S_j}$ is injective. Since $\dim(S_1+\cdots+S_m) \leq \dim(\col(F_1)+\cdots+\col(F_m))$, it suffices to show that the former is at least $m \cdot \rank(E_s)$.
	
	Let $\{v_1,\dots,v_m\}$ be any set of vectors such that $v_j \in S_j$ for all $j \in [m]$ and $v_1 + \cdots + v_m = 0$. We will show that necessarily $v_j = 0$ for all $j \in [m]$. To see this, pick $j_1 \in [m]$ such that $g(j_1)$ is minimal and observe that $g$ is necessarily injective by \eqref{eq:triangularcriterion}, so that this $j_1$ is uniquely defined. If we then apply $f_j$, or equivalently multiply on the left by $T_{g(j_1)}^\top$, we find by \eqref{eq:triangularcriterion}
	\begin{align*}
		0 &= T_{g(j_1)}^\top\left(v_1 + \cdots + v_m\right) \\
		&= T_{g(j_1)}^\top v_{j_1}.
	\end{align*}
	Since $f_{j_1}|_{S_{j_1}}$ is injective, it follows that $v_{j_1} = 0$. Now continue by picking $j_2 \in [m] \setminus \{j_1\}$ such that $g(j_2)$ is minimal. For the same reason, we find that $v_{j_2} = 0$. By repeating this process, we eventually find that $v_j = 0$ for all $j \in [m]$, which shows (b).
\end{proof}

It seems quite magical that such a $g$ could exist, and yet in the situations we consider this will be the case. It is especially surprising that we can simply take the matrix $T_h^\top$ in \eqref{eq:triangularcriterion} instead of needing more general linear combinations of the `type $i$'-adjacency matrices to make this step work. \\

We will now show how to verify the triangle criterion in practice.

\begin{lemma}\label{lem:checktrianglecriterion}
	Let $\{F_1,\dots,F_m\}$ be as defined before. The function $g:[m] \rightarrow [\ell]$ satisfies the triangle criterion if and only if for all $j \in [m]$ we have
	\begin{align*}
		&\sum_{i = 1}^\ell f_{ij} t_{hi}^k = 0 \text{ for all } k \in \{0,\dots,d\} \text{ and } h < g(j), \text{ and } \\
		&\sum_{i = 1}^\ell \sum_{k=0}^d f_{ij} t_{g(j),i}^kp_k(r) \begin{cases}
			\neq 0 & \text{if } r = s, \\
			= 0 & \text{if } r \neq s.
		\end{cases}
	\end{align*}
\end{lemma}

\begin{proof}
	Let $h \in [\ell]$ and consider the product $T^\top_h F_j$. We can expand $F_j$ in terms of its definition, and rewrite the products $T_h^\top T_i$ using \eqref{eq:expandproductsofTinassocscheme}:
	\begin{align}
		T^\top_h F_j &= T^\top_h \left(\sum_{i = 1}^\ell f_{ij}T_i\right) \notag\\
		&= \sum_{i = 1}^\ell \sum_{k=0}^d f_{ij} t_{hi}^kA_k. \label{eq:productTFinassocscheme}
	\end{align}
	It follows that the second requirement of the triangle criterion is satisfied if and only if $\sum_{i = 1}^\ell f_{ij} t_{hi}^k = 0$ for all $k \in \{0,\dots,d\}$ and $h < g(j)$, as the matrices $\{A_0,\dots,A_d\}$ are linearly independent. As for the first requirement, we will multiply the equation above for $h = g(j)$ with $E_r$ to find
	\begin{align*}
		T^\top_{g(j)} F_j &= \sum_{i = 1}^\ell \sum_{k=0}^d f_{ij} t_{g(j),i}^kp_k(r)E_r.
	\end{align*}
	It follows that $g$ satisfies the first requirement if and only if
	\begin{align*}
		\sum_{i = 1}^\ell \sum_{k=0}^d f_{ij} t_{g(j),i}^kp_k(r) \begin{cases}
			\neq 0 & \text{if } r = s, \\
			= 0 & \text{if } r \neq s.
		\end{cases}
	\end{align*}
\end{proof}	

Even though checking the requirements of \Cref{lem:checktrianglecriterion} may look daunting, for our choices of $J$, we have either $d \leq 2$ or $m = 1$. Moreover, the coefficients $f_{ij}$ will be powers of $q$ up to sign or zero, which will aid in simplifying the computations.

\subsection{Summary}

We can summarize the necessary computations in the following overview. For every module $M$ listed in \Cref{thm:multonmodules} we will do the following.

\begin{itemize}[{}]
	\item \textbf{Step 1.} Choose $J \subseteq [n]$ such that $\rho_M$ is a component of $\mathrm{ind}_{P_J}^G(1_{P_J})$ and define the type of a maximal flag. Compute $|\cC_i^X|$ for all $i \in [\ell]$.
	\item \textbf{Step 2.} Compute the quotient matrix $Q$ defined by $Q_{ij} = |\cC_j^X \cap \opp(c)|$ for arbitrary $c \in \cC_i^X$, $i,j \in [\ell]$. Find $m$ linearly independent eigenvectors for $\lambda_{\min}$.
	\item \textbf{Step 3.} Compute $t_{ij}^k=|\cC_i^X \cap \cC_j^Y|$, for all $i,j \in [\ell]$ and $X$ and $Y$ in relation $k \in \{0,\dots,d\}$. Find $g:[m]\rightarrow[\ell]$ with the triangular criterion through \Cref{lem:checktrianglecriterion}.
\end{itemize}

\begin{remark}
	The quantities $|\cC_i^X|$, $Q_{ij}$ and $t_{ij}^k$ can be computed in different ways. We will employ a geometrical approach, as EKR problems are more often studied by combinatorialists. However, we would be remiss to not mention other approaches. Abramenko, Parkinson and Van Maldeghem \cite{APVM17} developed a building-theoretical point of view to give formulae for these quantities. In this setting, $|\cC_i^X|$ is the cardinality of a generalized sphere in a building, for which \cite[Theorem 2.1]{APVM17} provides a general formula. In our case, it reduces to $|\cC_i^X| = q^{\ell(w_i)}\sum_{w \in W_J}q^{\ell{(w)}}$, which we have seen before. On the other hand, the quantities $Q_{ij}$ and $t_{ij}^k$ are called intersection numbers, as they are the cardinality of intersections of generalized spheres. In the notation of \cite{APVM17}, these are respectively $c^{w_i}_{w_j,w_0}(J,\emptyset,\emptyset)$ and $c_{w_j,w_i}^{k}(J,\emptyset,J)$, where $w_i,w_j$ are the coset representatives of minimal length of type $i$ and $j$ respectively, and $w_0$ is the longest word in $W$. They can be computed in a combinatorial way using galleries \cite[Theorem 3.7]{APVM17}. 
\end{remark}


\section{Application in each type} \label{sec:specifictypes}

Unfortunately, it seems that in order the apply the method outlined in the previous section, we have to treat each case separately. We will do so in the following few subsections, resulting in the main results for type $A_n$ (\Cref{cor:typeAmainresult}), type $B_n$ (\Cref{cor:typeBmainresult} and \Cref{sec:typeBhardmodule}) and type $D_n$ (\Cref{cor:typeDmainresult}).

\subsection{Type $A_n$}

Let us reiterate the goal in type $A_n$. The eigenspace of $A_{w_0}$ corresponding to $\lambda_{\min}$ is contained in $M_{[n,1]}$ (\Cref{maintheoremprev}) and has dimension $\lceil n/2 \rceil(q^{n+1}-q)/(q-1)$ (\Cref{thm:multonmodules}). \\

We will choose $J = [n]\setminus\{1\}$, so that the parabolic subgroup $P_J$ is the stabilizer of a point in $\PG(n,q)$. This choice is motivated by several reasons. First of all, as we saw before, the number of orbits on maximal flags equals $n+1$, which is minimal among all choices of $J \subsetneq [n]$. Secondly, the association scheme on points is very simple, being the $1$-class trivial scheme. This is encoded by the action of $P_J$ on the set of points, which clearly has two orbits: the fixed point itself and all other points. Equivalently, the association scheme has only two basis relations: identity and non-identity. We can also see this algebraically in the decomposition of the corresponding permutation character $1^G_{P_J}$, or equivalently on the level of the corresponding Weyl group (by the results of Curtis, Iwahori and Kilmoyer \cite{CIK71}). The decomposition in the Weyl group $\Sym(n+1)$ can be computed by the branching rule \cite[Section 6.1.8]{GP00} and is \[\mathrm{ind}_{W_J}^W(1_{W_J}) = \chi_{[n+1]} + \chi_{[n,1]}.\] 

Now, the appearance of $\chi_{[n,1]}$ is necessary for the method to work, but it does not motivate our particular choice of $J$: the multiplicity of $\chi_{[n,1]}$ in $1^G_{P_J}$ for general $J \subseteq [n]$ is $n-|J|$ \cite[Remark 6.3.7]{GP00}. So from this point of view, any choice for $J$ would be good.

The final reason for our choice of $J$ is that the automorphism groups of the known maximal examples when $n$ is odd are point stabilizers (or their duals) and we will see that these examples can be easily described in terms of the resulting eigenvectors. \\

\noindent \textbf{Step 1.} Let $J = [n] \setminus \{1\}$. Then $G/P_J$ can be identified with the points of $\PG(n,q)$ and we will denote this set as $\cP$. 

\begin{notation}
	We will denote the number of points and maximal flags in $\PG(n,q)$ by $v(n)$ and $c(n)$ respectively. We suppress $q$ from this notation as it is assumed to be fixed throughout. Observe that
	\begin{align*}
		&|\cP| = v(n) = \gauss{n+1}{1} \,\, \text{ and }\,\, |\cC| = c(n) =\prod_{i=1}^n v(i).
	\end{align*}
\end{notation}

One can check that the minimal length coset representatives of $W_J \backslash W$ are $w_1 := 1$ and $w_i := s_1s_2\cdots s_{i-1}$, $i \in \{2,\dots,n+1\}$.
By \eqref{eq:sizeofCiX}, we immediately deduce that for any $X \in \cP$
\begin{align}\label{eq:typeAsizeofCiX}
	|\cC_i^X| = c(n-1)q^{i-1},
\end{align}
as $W_J$ is a Weyl group of type $A_{n-1}$. \\

We can describe the type of a maximal flag geometrically. Recall that $\dim(U)$ always refers to the vector dimension or rank of a subspace $U \subset \PG(n,q)$. We usually write maximal flags as $(U_1,\dots,U_n)$, where $U_i \subseteq U_{i+1}$ and $\dim(U_i) = i$ for all $i \in [n]$, with the understanding that $U_{n+1}$ denotes the whole space.
 
\begin{defin}\label{def:typeofflag}
	Let $c = (U_1,\dots,U_{n})$ be a maximal flag in $\PG(n,q)$, and $X \in \cP$. The \textbf{type} of $c$ w.r.t.\ $X$ is the smallest index $k$ such that $X \in U_k$ and $n+1$ otherwise.
\end{defin}	

This is the specialization of the more general \Cref{def:typegeneral} applied to the aforementioned choices of $G$, $P_J$. Recall that the Weyl group is $W = \Sym(n+1)$ and so the corresponding subgroup $W_J$ is the stabilizer of the element $1$. The coset representatives can then be written as $w_k = s_1s_2\dots s_k$, where $s_i$ is the transposition permuting $i$ and $i+1$. The geometric interpretation in \Cref{def:typeofflag} can then be obtained by looking at the action of $w_i$ on the standard maximal flag $(U_1,\dots,U_n)$ where $U_j = \mathrm{span}(e_1,\dots,e_j)$ and taking the orbit of this flag under $B$. \\

\noindent \textbf{Step 2.} Now let $Q$ be the $(n+1) \times (n+1)$ quotient matrix of the opposition graph with respect to the action of $P_J$ on $\cC$. Recall that the value $Q_{ij}$ is the number of flags of type $j$ opposite to a given flag of type $i$.

\begin{lemma}\label{lem:typeAquotientmatrix}
	The entries of the quotient matrix are
	\begin{align*}
		Q_{ij}=\begin{cases}
		0 & \text{ if } i+j < n+2\\
		(q-1+\delta_{i+j,n+2})q^{\frac{1}{2}(n^2-n)+j-2} & \text{ if } i+j \geq n+2.
	\end{cases}
	\end{align*}
\end{lemma}
\begin{proof}
	Fix a maximal flag $c=(U_1,\dots,U_n)$ and let $S := \{(Y,d) \,\, | \,\, d \in \cC_j^Y \cap \opp(c) \text{ and } c \in \cC_i^Y\}$. Since $Y \in U_i \setminus U_{i-1}$, there are $q^{i-1}$ choices for $Y$ and by symmetry they all appear in the same number $x$ of pairs. Hence $|S|=xq^{i-1}$. 
	
	On the other hand, there are $q^{n+1\choose 2}$ maximal flags $d=(V_1,\dots,V_n)$ that are opposite to $c$ and for a point $Y$ we have $(Y,d)\in S$ if and only if $Y\in U_i\setminus U_{i-1}$ and $Y\in V_j\setminus V_{j-1}$. If $i+j\le n+1$, then $U_i\cap V_j=\emptyset$ and hence $S=\emptyset$. In this case $x=0$. If $i+j=n+2$, then $U_i\cap V_j$ is a point $Y$ with $Y\notin U_{i-1}\cup V_{j-1}$. Hence, if $i+j=n+2$, then $d$ occurs in exactly one pair of $S$, so $|S|=q^{n+1\choose 2}$ and thus $xq^{i-1}=q^{{n+1\choose 2}}$ giving $x=q^{\frac{1}{2}(n^2-n)+j-1}$ as $i=n+2-j$. Finally, assume that $i+j > n+2$. Then the number of points $Y$ that are contained in $U_i\setminus U_{i-1}$ and in $V_j\setminus V_{j-1}$ is equal to
	\begin{align*}
		|U_i\cap V_j|-|U_i\cap V_{j-1}|-|U_{i-1}\cap V_j|+|U_{i-1}\cap V_{j-1}|
		=q^{i+j-n-2}(q-1)
	\end{align*}
	so $xq^{i-1}=q^{n+1\choose 2}q^{i+j-n-2}(q-1)$ giving $x=(q-1)q^{\frac{1}{2}(n^2-n)+j-2}$.
\end{proof}

\begin{exa}
	For $n = 3$ we find the quotient matrix
	\[\begin{pmatrix}
		0 & 0 & 0 & q^6 \\
		0 & 0 & q^5 & q^6-q^5 \\
		0 & q^4 & q^5-q^4 & q^6-q^5 \\
		q^3 & q^4-q^3 & q^5-q^4 & q^6-q^5 
	\end{pmatrix}\]
	with eigenvalues $q^6$, $q^4$ and $-q^4$ twice. Corresponding eigenvectors are $(1,1,1,1)$, $(q^2,-q,-q,1)$, $(0,q,-1,0)$ and $(q^2,q^2,-1,-1)$. These are clearly the eigenvalues coming from the characters $\chi_{[n+1]}$ and $\chi_{[n,1]}$, and the eigenvectors possess some structure that allow us to generalize.
\end{exa}

We will now restrict ourselves to the case when the rank $n$ is odd. Even though computations for even $n$ are similar, the smallest eigenvalue in this case is not an integer power of $q$. It follows that the bound on the size of sets of non-opposite flags obtained in \cite{DBMM22} is not necessarily an integer for all $q$ and hence can not be sharp for those values of $q$. We therefore have no use for a basis of the eigenspace, as there will not be a possibility for a classification of the maximal EKR-sets. For this reason \textbf{we assume for the rest of this section that $n$ is odd}.

\begin{theorem}
	Write $n = 2m-1$. For $j \in [m]$ define the column vector
	\[v_j := (\underbrace{0,\dots,0}_{m-j},\underbrace{q^j,\dots,q^j}_{j},\underbrace{-1,\dots,-1}_{j},\underbrace{0,\dots,0}_{m-j})^\top.\]
	Then $v_j$ is an eigenvector of $Q$ with eigenvalue $-q^{(n^2-1)/2}$.
\end{theorem}

\begin{proof}
	Let $w_1,\dots,w_{n+1}$ be the row vectors of $Q$. We have to show that $w_kv_j = -q^{(n^2-1)/2}(v_j)_k$ for all  
	$1 \leq j,k \leq n+1$, where $(v_j)_k$ is the $k$-th entry of $v_j$. From \Cref{lem:typeAquotientmatrix} we have 
	\[w_k = q^{\frac{1}{2}n(n+1)} \cdot (\underbrace{0,\dots,0}_{n+1-k},q^{-k+1},q^{-k+1}(q-1),q^{-k+2}(q-1),\dots,q^{-1}(q-1)).\]
	
	For two non-negative integers $a \leq b$, we will use that $\sum_{i = a}^b q^{-i} = \frac{q^{b-a+1}-1}{q^b(q-1)}$ throughout. 
	
	\noindent Case 1. $k \leq m-j$. Then $w_kv_j = 0$ and $(v_j)_k = 0$.
	
	\noindent Case 2. $m - j < k \leq m$. Then $(v_j)_k = q^j$ and
	\[w_kv_j = q^{\frac{1}{2}n(n+1)}\left(-q^{-k+1}-(q-1)\sum_{i=m-j+1}^{k-1}q^{-i}\right) = -q^{\frac{1}{2}n(n+1)-m+j} = -q^{(n^2-1)/2} \cdot q^j,\]
	where the summation is to be understood as $0$ if $k = m-j+1$.
	
	\noindent Case 3. $m < k \leq m + j$. Then $(v_j)_k = -1$ and
	\begin{align*}
		w_kv_j = q^{\frac{1}{2}n(n+1)}\left(q^{j-k+1}+q^j(q-1)\sum_{i=m+1}^{k-1}q^{-i}-(q-1)\sum_{i = m-j+1}^{m}q^{-i}\right) = -q^{(n^2-1)/2} \cdot (-1),
	\end{align*}
	where the second summation is to be understood as $0$ if $k = m+1$.
	
	\noindent Case 4. $m + j < k$. Then $(v_j)_k = 0$ and 
	\[w_kv_j = q^{\frac{1}{2}n(n+1)}\left(q^j(q-1)\sum_{i=m+1}^{m+j}q^{-i}-(q-1)\sum_{i = m-j+1}^{m}q^{-i}\right) = 0. \qedhere\]
\end{proof}

As described in \Cref{subsection:step2}, we can now define matrices whose columns are eigenvectors of $A_{w_0}$.

\begin{defin}\label{def:typeAmatrixofeigenvectors}
	For each $j \in [m]$ we can define the $|\cC| \times |\cP|$ matrix $F_j$ as
	\[F_j = q^j\sum_{i=m-j+1}^{m} T_i-\sum_{i=m+1}^{m+j} T_i,\]
	where $T_i$ denotes the type $i$ adjacency matrix.
\end{defin}

\noindent \textbf{Step 3.} We will now prove that the columns of $F_j$, $j \in [m]$, span the $\lambda_{\min}$-eigenspace of $A_{w_0}$.

\begin{lemma}\label{lem:typeAbetas} 
	For all $i,j \in [n+1]$ we have
	\[|\cC_i^X \cap \cC_j^Y| = \begin{cases}
		\delta_{ij}c(n-1)q^{i-1} & \text{if } X = Y \\
		c(n-2)q^{i-2}(q^{j-1}-\delta_{ij}) & \text{if } X \neq Y.
	\end{cases}\]
\end{lemma}
\begin{proof}
	When $X = Y$, we immediately see that $|\cC_i^X \cap \cC_j^Y| = \delta_{ij}|\cC_i^X|$ as a flag has only one type and we know this from \eqref{eq:typeAsizeofCiX}.
	
	So fix $X \in \cP$. We will count the pairs $(Y,c)$ such that $X \neq Y$ and $c \in \cC_i^X \cap \cC_j^Y$ in two ways. There are $|\cC_i^X|=c(n-1)q^{i-1}$ possibilities for $c$ for every such $c$ there are $q^{j-1}-\delta_{ij}$ ways to pick $Y$. The $\delta_{ij}$ appears since $Y$ has to be distinct from $X$. On the other hand, for every of the $v(n)-1 = qv(n-1)$ possibilities for $Y$, there are equally many possibilities for $c$ by symmetry, namely $|\cC_i^X \cap \cC_j^Y|$ many. We conclude that 
	\[|\cC_i^X \cap \cC_j^Y|\cdot qv(n-1) = c(n-1)q^{i-1}(q^{j-1}-\delta_{ij}),\]
	from which the lemma follows.
\end{proof}

Recall that the association scheme on $\cP$ is trivial: the basis matrices $\{A_0,A_1\}$ are $A_0 = I_\cP$ and $A_1 = J_\cP-I_\cP$. Consequently, the minimal idempotents are $E_0 = \frac{1}{|\cP|}J_\cP$ and $E_1 = I_\cP - \frac{1}{|\cP|}J_\cP$. 
In this case, the $P$-matrix is 
\[\begin{pmatrix}
	1 & v(n)-1 \\
	1 & -1
\end{pmatrix}\]
and $E_s = E_1$ is the special idempotent corresponding to the module indexed by $[n,1]$. With this in hand, we are now in the position to state the main theorem of this section.

\begin{theorem}
	Define $g:[m]\rightarrow[m]$ by $g(j):=m-j+1$. Then $g$ satisfies the triangular criterion.
\end{theorem}

\begin{proof}
	We will do so by checking the conditions in \Cref{lem:checktrianglecriterion} for all $j \in [m]$, which are
	\begin{align*}
		&\sum_{i = 1}^{n+1} f_{ij} t_{hi}^k = 0 \text{ for all } k \in \{0,1\} \text{ and } h < m-j+1 \text{ and } \\
		&\sum_{i = 1}^{n+1} \sum_{k=0}^1 f_{ij} t_{m-j+1,i}^kp_k(r) \begin{cases}
			= 0 & \text{if } r = 0, \\
			\neq 0 & \text{if } r = 1. 			
		\end{cases}
	\end{align*}

	We have the following information:
	\begin{enumerate}[(1)]
		\item By \Cref{def:typeAmatrixofeigenvectors} we have for all $j \in [m]$
		\[f_{ij} = \begin{cases}
			q^j & \text{if } m-j+1 \leq i \leq m, \\
			-1 & \text{if } m+1 \leq i \leq m+j, \\
			0 & \text{otherwise.}
		\end{cases}\]
		
		\item By \Cref{lem:typeAbetas} we have for all $i,j \in [{n+1}]$
		\[\begin{cases}
			t_{ij}^0 = \delta_{ij}c(n-1)q^{i-1} \\
			t_{ij}^1 = c(n-2)q^{i-2}(q^{j-1}-\delta_{ij}).
		\end{cases}\]
		\item The $P$-matrix of $\cA(G,G/P_J)$ is 
		\[P = \begin{pmatrix}
			1 & v(n)-1 \\
			1 & -1
		\end{pmatrix}\]
		\item The previous two points together imply
		\begin{align}\label{eq:typeAPtimesTcoeffs}
			\begin{pmatrix} t^0_{ij}p_0(0)+t^1_{ij}p_1(0) \\ t^0_{ij}p_0(1)+t^1_{ij}p_1(1) \end{pmatrix}
		= P
		\begin{pmatrix} t^0_{ij} \\ t^1_{ij} \end{pmatrix} 
		= \begin{pmatrix} c(n-1)q^{i+j-2} \\ c(n-2)\left(\delta_{ij}v(n)q^{i-2}-q^{i+j-3}\right) \end{pmatrix}
		\end{align}
	\end{enumerate}

	The first condition can now be rewritten as
	
	\[\sum_{i = m-j+1}^m q^j (t_{hi}^0+t_{hi}^1) - \sum_{i = m+1}^{m+j} (t_{hi}^0+t_{hi}^1) = 0 \text{ for all } h < m-j+1 \]
	The indicator $\delta_{hi}$ appearing in $t_{hi}^0$ and $t_{hi}^1$ will never be $1$ as $h < m-j+1 \leq i$, so that the first condition is trivially true for $k = 0$ and simplifies for $k = 1$ to
	\[c(n-2)\left(\sum_{i = m-j+1}^m q^{h+i+j-3} - \sum_{i = m+1}^{m+j} q^{h+i-3}\right) = 0,\]
	which is indeed true.
	
	Using $g(j) = m-j+1$ and the first row of \eqref{eq:typeAPtimesTcoeffs}, the second expression for $r=0$ simplifies to
	\[c(n-1)\sum^{n+1}_{i=1}f_{ij}q^{i+m-j-1} = c(n-1)\left(\sum_{i = m-j+1}^m q^{i+m-1} - \sum_{i = m+1}^{m+j} q^{i+m-j-1}\right)\]
	which is indeed zero.
	
	On the other hand, using the second row of \eqref{eq:typeAPtimesTcoeffs}, we find for $r=1$ that
	\[c(n-2)\sum_{i=1}^{n+1} f_{ij} \left(\delta_{i,m-j+1}v(n)q^{i-2}-q^{i+m-j-2}\right).\]
	The second term is again zero, while the contribution of the first term is $v(n)c(n-2)q^{m-1} \neq 0$.	
%
%
\end{proof}

\begin{cor}\label{cor:typeAmainresult}
	The $\lambda_{\min}$-eigenspace for $A_{w_0}$ in type $A_{2m-1}$, $m \geq 1$, is $\col(F_1)\oplus \cdots \oplus \col(F_m)$.
\end{cor}

\begin{proof}
	By \Cref{lem:triangularcriterion}, the space $\col(F_1)+ \cdots + \col(F_m)$ has dimension at least $d := m \cdot \frac{q^n-q}{q-1}$, while the eigenspace has dimension exactly $d$ by \Cref{thm:multonmodules}. It follows that the sum is direct and the two spaces coincide.
\end{proof}	

\begin{remark}
	As observed before, since the generic degree equals $|\cP|-1$ and $F_j$ has $|\cP|$ columns, we can actually extract a basis by taking the union of all columns but one of every $F_j$, $j \in [m]$.
\end{remark}

\begin{remark}
	With a little more work, our results in this section hold in more generality. If we consider the opposition graph on flags of cotype $J$, where $J = J^{w_0}$, then we can use the same method to obtain a spanning set for the eigenspace corresponding to the smallest eigenvalue. We will assume $m \notin J$, since otherwise we do not have examples of maximal EKR-sets attaining the upper bound \cite[Theorem 3.8]{DBMM22}. The multiplicity ${\rm mult}_{[n,1]} = (n+1)/2 = m$ will be replaced by $(n+1-|J|)/2$ and the possible types of a partial flag with respect to a fixed point will shrink from $n+1$ options to $n+1-|J|$ options. We can define the matrix $F_j$ from \Cref{def:typeAmatrixofeigenvectors} for all $j \in [m] \setminus J$ and continue as before to obtain a set of spanning vectors (and even a basis) of the eigenspace of the smallest eigenvalue.
	
\end{remark}
	
\begin{remark}
	Recall that there are two natural maximal EKR-sets of flags in $\PG(2m-1,q)$, conjectured to be the only ones \cite{DBMM22}. These are the sets of all flags $(U_1,\dots,U_{2m-1})$ such that $U_m$ contains a fixed point or dually, is contained in a fixed hyperplane. We can call these maximal EKR-sets of `point-type' and of `hyperplane-type'. One could have started by choosing $J = [n] \setminus \{n\}$. This leads to the same computations and gives eigenvectors which are constants on the orbits of $P_J$, which now is the stabilizer of a hyperplane. This reflects algebraically the two families of maximal EKR-sets of flags. 
	
	With the choice we made above, the EKR-sets of point-type are easily described. Let $S$ be such a set with respect to the point $X$, then $\frac{|\cC|}{|S|}\textbf{1}_S = \textbf{1}_\cC + (F_m)_X$, where $(F_m)_X$ denotes the column of $F_m$ indexed by $X$. 
	
	Interestingly, we can partition the types depending on whether it is at most $m$, or at least $m+1$. In this way, we can take a $2 \times 2$ quotient matrix of $Q$ to obtain the matrix
	\[Q' = \begin{pmatrix}
		0 & q^{m(2m-1)} \\ q^{2m(m-1)} & q^{m(2m-1)}-q^{2m(m-1)} 
	\end{pmatrix}.\]
	Its eigenvector for $\lambda_{\min}$ is $(q^m,-1)^\top$ and will lift to the same eigenvectors as the columns of $F_m$, whose importance we described the paragraph before.
\end{remark}

\subsection{Type $B_n$: module $M_{([n-1],[1])}$}

Fix a polar space $\PS(n,e,q)$ in \Cref{table:classicalgroups} of rank $n$ and type $e$ defined over $\Fq$. We will again define a quotient matrix of its opposition graph on maximal flags by considering the stabilizer of point $X$. The motivations for this choice are the same as in type $A_n$: low number of orbits on $\cC$, well-understood association scheme on the points and instances of maximal EKR-sets with this automorphism group. This time the action of the parabolic subgroup $P_J$, $J = [n] \setminus \{1\}$, has three orbits on points: $\{X\}$, the points collinear to $X$ and the points opposite to $X$. The corresponding permutation character decomposes as follows (see the branching rule for type $B_n$ \cite[6.1.9]{GP00}):
\[\mathrm{ind}^W_{W_J}(1_{W_J}) = \chi_{([n],\emptyset)} + \chi_{([n-1,1],\emptyset)} + \chi_{([n-1],[1])}.\]

We conclude that this choice has the potential to give us the eigenvectors contained in the module indexed by $([n-1],[1])$.

\noindent \textbf{Step 1.} Let $J = [n] \setminus \{1\}$. Then $G/P_J$ can be identified with the points of $\PS(n,e,q)$ and we will denote this set again as $\cP$. 

As before, we will consider the orbits of $P_J$ on $\cC$. If $U$ is a subspace contained in $\PS(n,e,q)$, then as usual, $U^\perp$ denotes the set of points that are collinear to all points in $U$. A classical fact from the theory of polar spaces states that if $U$ has rank $i$, then the quotient space $U^\perp / U$ is isomorphic to $\PS(n-i,e,q)$. Recall that the rank of a subspace in $\PS(n,e,q)$ equals its vector dimension.

\begin{defin}
	Let $c = (U_1,\dots,U_n)$ be a maximal flag in $\PS(n,e,q)$ and $X \in \cP$. The \textbf{type} of $c$ w.r.t.\ $X$ is the smallest index $k$ such that $X$ is contained in the $k$-th subspace in the chain $U_1\subseteq \dots \subseteq U_n\subseteq U_{n-1}^\perp \subseteq \dots\subseteq U_1^\perp$ and $2n$ otherwise. 
\end{defin}

Throughout we will use that if $c = (U_1,\dots,U_n)$ is of type $i$ w.r.t.\ $X$ and $i > n$, then this is equivalent to $X \in U_{2n-i}^\perp \setminus U_{2n-i+1}^\perp$. 

\begin{lemma}\label{lem:typeBsizeofCiX}
	For all $i \in [2n]$ and $X \in \cP$ we have 
	\begin{align*} 
		|\cC_i^X| = \begin{cases}
			c(n-1,e)q^{i-1} &\text{if } i\le n, \\
			c(n-1,e)q^{i+e-2} &\text{if } i>n.
		\end{cases}
	\end{align*}
\end{lemma}
\begin{proof}
	First assume that $i \leq n$. Then consider all pairs $(d,Y)$ of maximal flags $d=(V_1,\dots,V_n)$ and points $Y$ such that $Y\in V_{i}\setminus V_{i-1}$. Each maximal flag occurs in $q^{i-1}$ such pairs, so there are $c(n,e)q^{i-1}$ pairs. By symmetry all points $Y$ occur in the same number of pairs and hence every point occurs in $c(n,e)q^{i-1}/v(n,e)=c(n-1,e)q^{i-1}$ pairs.
	
	On the other hand if $i > n$, we consider all pairs $(d,Y)$ of maximal flags $d=(V_1,\dots,V_n)$ and points $Y$ such that $Y\in V_{2n-i}^\perp\setminus V_{2n-i+1}^\perp$. Fix a flag $d$ and consider the quotient space $V_{2n-i}^\perp/V_{2n-i} \cong \PS(i-n,e,q)$. We observe that $V_{2n-i+1}/V_{2n-i}$ is a point $Z$ in the quotient space and every point in the quotient space opposite to $Z$ corresponds to $q^{2n-i}$ points $Y$ with $Y\in V_{2n-i}^\perp\setminus V_{2n-i+1}^\perp$. This implies that each maximal flag occurs in $d(i-n,e,1)q^{2n-1} = q^{i+e-2}$ pairs, so there are $c(n,e)q^{i+e-2}$ pairs. By symmetry all points $Y$ occur in the same number of pairs and hence every point occurs in $c(n,e)q^{i+e-2}/v(n,e)=c(n-1,e)q^{i+e-2}$ pairs.
\end{proof}

We can transfer the notation from the previous section and collect some well-known facts about polar spaces and introduce some notation for the results to come. In the rest of this section we fix a finite classical (non-degenerate) polar spaces $\PS(n,e,q)$ of rank $n \geq 2$ and type $e$ defined over $\Fq$. We will therefore omit $q$ mostly as an index or parameter.

\begin{lemma}\label{lem:typeBnotation}
	Let $\PS(n,e,q)$ be a polar space.
	\begin{enumerate}\renewcommand{\labelenumi}{(\alph{enumi})}
		\item The number of points in $\PS(n,e,q)$ is
		\begin{align*}
			v(n,e):=\frac{q^n-1}{q-1}(q^{n-1+e}+1).
		\end{align*}
		\item The number of maximal flags in $\PS(n,e,q)$ is
		\begin{align*}
			c(n,e):=\prod_{i=1}^n v(i,e).
		\end{align*}
		
		\item The number of points opposite to two points $X$ and $Y$ is
		\begin{align*}
			\begin{cases}
				\alpha(n,e):=q^{2n+e-2} & \text{if }  X=Y,
				\\
				\beta(n,e):=(q-1)q^{2n+e-3}+q^{n-2}(q^e-q) & \text{if } $X$ \text{ and } $Y$ \text{ are opposite,}
				\\
				\gamma(n,e):=(q-1)q^{2n+e-3} & \text{otherwise}.
			\end{cases}
		\end{align*}
		\item Every rank $i$ subspace is opposite to $d(n,e,i):=q^{2i(n-i)+ie+{i\choose 2}}$ rank $i$ subspaces.
	\end{enumerate}
\end{lemma}
\begin{proof}
	(a) and (b) can be deduced from \cite[Lemma 9.4.1]{BCN89}. 
	
	(c) can be seen from the fact that the opposition graph on points of $\PS(n,e,q)$ is strongly regular. It is the complement of the collinearity graph, whose parameters can be found from \cite[Theorem 2.2.12]{BVM22}
	
	(d) is the valency of the opposition graph on rank $i$ subspaces. This value can for example be found in \cite[Proposition 3.1]{Brouwer10} or \cite[Proposition 3.22]{DBMM22}.
\end{proof}

\noindent \textbf{Step 2.} Now let $Q$ be the $2n \times 2n$ quotient matrix of the opposition graph with respect to the action of $P_J$. Recall that the value $Q_{ij}$ is the number of flags of type $j$ opposite to a given flag of type $i$.

Next we prepare some calculations for the computation of the quotient matrix. The double counting proof as in type $A_n$ does not work since the number of points of $\cP$ contained in $(U_{2n-i}^\perp \setminus U_{2n-i+1}^\perp) \cap (V_{2n-j}^\perp \setminus V_{2n-j+1}^\perp)$ is not so straightforward to count. Instead we will employ a direct method: given a point $X$ and $c \in \cC_i^X$ we will construct all flags $d \in \cC_j^X$ opposite to $c$ and count the number of ways we could have done so. The following lemma contains the crux of these computations.

\begin{lemma}\label{lem:typeBconstructionofoppositeflag}
	Let $1\le i\le 2n$ and $n < j \leq 2n$ with $i+j \geq 2n+1$, $X \in \cP$ and $c=(U_1,\dots,U_n) \in \cC^X_i$.
	\begin{enumerate}\renewcommand{\labelenumi}{(\alph{enumi})}
		\item The number of $(2n-j)$-subspaces $A$ with $X\in A^\perp$ and $A^\perp\cap U_{2n-j}=\emptyset$ is $d(n-1,e,2n-j) q^{2n-j}$.
		\item Given a subspace $A$ as in (a), then the number of $(2n-j+1)$-subspaces $B$ with $A\subseteq B$ and $X\notin B^\perp$ and $B^\perp\cap U_{2n-j+1}=\emptyset$ is
		\[\begin{cases}
			\alpha(j-n,e) & \text{if } i+j=2n+1, \\
			\beta(j-n,e) & \text{if } i=j, \\
			\gamma(j-n,e) & \text{otherwise.}
		\end{cases}\]
		This number equals $(q-1+\delta_{i+j,2n+1})q^{2(j-n)+e-3}+\delta_{ij}q^{j-n-2}(q^e-q)$.
		\item Given  $A,B$ as in (b), then $\{A,B\}$ extends to $q^{{2n-j \choose 2} + (j-n-1)(j-n+e-2)}$ maximal flags opposite to $c$.
	\end{enumerate}
\end{lemma}
\begin{proof}
	(a)\, Since $i>2n-j$, we see that $X\notin U_{2n-j}$ and $X\in U_{2n-j}^\perp$. Hence if $A$ is a subspace as wanted, the subspaces $\erz{X,U_{2n-j}}/X$and $\erz{X,A}/X$ are opposite subspaces of rank $2n-j$ in the quotient space $X^\perp/X \cong \PS(n-1,e,q)$. In this polar space, the number of $(2n-j)$-subspaces $C/X$ opposite to $\erz{X,U_{2n-j}}/X$ is $d(n-1,e,2n-j)$. For any such $C$, we have that $C$ is a subspace of rank $2n-j+1$ of the original polar space. It follows that there are $q^{2n-j}$ hyperplanes in $C$ not containing $X$, each of which is a suitable possibility for $A$. In conclusion we have $d(n-1,e,2n-j) q^{2n-j}$ possibilities for $A$ in total. \\
	
	(b)\, Since $B^\perp\subseteq A^\perp$ and $A^\perp \cap U_{2n-j}=\emptyset$, we see that $Y:=U_{2n-j+1}\cap A^\perp$ is a point. The requirement $B^\perp\cap U_{2n-j+1}=\emptyset$ then means that $B$ is not contained in $Y^\perp$. Hence, we look for $(2n-j+1)$-subspaces $B$ with $A\subseteq B\subseteq A^\perp$ and $B\not\subseteq X^\perp\cup Y^\perp$. In the quotient space $A^\perp/A \cong \PS(j-n,e,q)$, we have that $B/A$ is a point, that is opposite to the points $\erz{A,X}/A$ and $\erz{A,Y}/A$. We can now count the possibilities for $B$ using \Cref{lem:typeBnotation}(c) depending on the relation between $X$ and $Y$.
	
	\begin{compactenum}[-]
	\item If $i+j=2n+1$, then $U_{2n-j+1}=U_i$ and hence $X=Y$, so the number of subspaces $B$ is $\alpha(j-n,e)$;
	
	\item if $i+j>2n+1$ and $i \leq n$, then $X \in U_i \setminus U_{i-1}$ and $Y \in U_{2n-j+1} \subseteq U_{i-1}$ implies that $X$ and $Y$ are distinct and collinear points, and hence there are $\gamma(j-n,e)$ possibilities for $B$;
	
	\item if $i+j>2n+1$ and $n < i = j$, then  $X\in U_{2n-j}^\perp\setminus U_{2n-j+1}^\perp$ implies that $X$ and $Y$ are opposite points, and hence there are $\beta(j-n,e)$ possibilities for $B$;

	\item if $i+j>2n+1$ and $n < i < j$, then $X \in U_{2n-i}^\perp \supseteq U_{2n-j+1}^\perp$ implies that $X$ and $Y$ are distinct and collinear points, and hence there are $\gamma(j-n,e)$ possibilities for $B$.
	\end{compactenum} 
	The last statement now follows from \Cref{lem:typeBnotation}.

	(c) By \cite[Corollary 3.1]{Brouwer10} the number of maximal flags extending $(A,B)$ and opposite to $c$ is $q^\ell$, where $\ell$ is the length of the longest word in $W_I$, with $I = \{1,\dots,2n-j-1,2n-j+2,\dots,n\}$. It follows that $W_I \cong \Sym(2n-j) \times W_{B_{j-n-1}}$ and hence $\ell = {2n-j \choose 2} + (j-n-1)(j-n+e-2)$.
\end{proof}

\begin{prop}\label{prop:typeBquotientmatrix}
	The entries of the quotient matrix are
	\begin{align*}
		Q_{ij}/q^{n(n+e-3)-e}= \begin{cases}
			0 & \mbox{if\ }\, i+j<2n+1,
			\\
			(q-1+\delta_{i+j,2n+1})q^j & \text{if }\, i+j\ge 2n+1,\ j\le n,
			\\
			(q-1+\delta_{i+j,2n+1})q^{j+e-1}+\delta_{ij}q^n(q^e-q) & \text{if }\, i+j>2n+1,\ j>n.
		\end{cases}
	\end{align*}
\end{prop}

\begin{proof}
	Fix a maximal flag $c=(U_1,\dots,U_n) \in \cC_i^X$ so that $Q_{ij}$ equals the number of maximal flags $d = (V_1,\dots,V_n) \in \cC_j^X$ opposite to $c$. 
	
	Assume first that $j > n$ so that $X \in V_{2n-j}^\perp \setminus V_{2n-j+1}^\perp$. As $c$ and $d$ are opposite it follows that $X \notin U_{2n-j}$ which implies that $Q_{ij} = 0$ if $i \leq 2n-j$. So assume $i+j \geq 2n+1$. We will now compute $Q_{ij}$ by building $d$, first by counting the possibilities for $V_{2n-j}$ and $V_{2n-j+1}$, ensuring that $d \in \cC_j^X$, and finally extending the partial flag $(V_{2n-j},V_{2n-j+1})$ to a maximal flag opposite to $c$. The computations have been done in \Cref{lem:typeBconstructionofoppositeflag} and we find
	\begin{align*}
		Q_{ij} = & \, d(n-1,e,2n-j) q^{2n-j} \cdot \\ &\cdot \left((q-1+\delta_{i+j,2n+1})q^{2(j-n)+e-3}+\delta_{ij}q^{j-n-2}(q^e-q)\right) \cdot q^{{2n-j \choose 2} + (j-n-1)(j-n+e-2)} \\
		= & \,(q-1+\delta_{i+j,2n+1})q^{n(n+e-3)+j-1}+\delta_{ij}q^{n(n+e-2)-e}(q^e-q).
	\end{align*}

	Assume now that $j \leq n$. Similarly as before, we have $X \in V_j$ and hence $X \notin U_j^\perp$, which implies that $Q_{ij} = 0$ if $i \leq 2n-j$. So we can assume that $i+j \geq 2n+1$ and hence $i > n$. By the double counting property \eqref{eq:quotientmatrixdoublecounting} of a quotient matrix, we know $|\cC_i^X|Q_{ij} = |\cC_j^X|Q_{ji}$. By \Cref{lem:typeBconstructionofoppositeflag}(a) and our previous computation, we have
	\begin{align*}
	Q_{ij} &= (q-1+\delta_{i+j,2n+1})q^{n(n+e-3)+i-1} \cdot q^{j-i-e+1} \\
	&= (q-1+\delta_{i+j,2n+1})q^{n(n+e-3)+j-e}. \qedhere
	\end{align*}
\end{proof}

\begin{exa}
	For $n = 2$ we find the quotient matrix
	\[\begin{pmatrix}
		0 & 0 & 0 & q^{2e+2} \\
		0 & 0 & q^{2e+1} & q^{2e+2}-q^{2e+1} \\
		0 & q^{e+1} & q^{2e+1}-q^{e+1} & q^{2e+2}-q^{2e+1} \\
		q^e & q^{e+1}-q^e & q^{2e+1}-q^{2e} & q^{2e+2}-q^{2e+1}+q^{2e}-q^{e+1}
	\end{pmatrix}\]
	with eigenvalues $q^{2e+2}$, $q^{2e}$ and $-q^{e+1}$ twice. Corresponding eigenvectors are $(1,1,1,1)$, $(q^2,-q,-q,1)$, $(0,q^e,-1,0)$ and $(q^{e+1},q^{e+1},-1,-1)$. We see the eigenvalues coming from the characters $\chi_{([n],\emptyset)}$, $\chi_{([n-1,1],\emptyset)}$ and $\chi_{([n-1],[1])}$, which are those appearing in the decomposition of the permutation character discussed earlier. The eigenvectors possess similar structure as in type $A_n$.
\end{exa}

\begin{lemma}
	For $i\in\{1,\dots,n\}$ define the column vector
	\[
	v_i:=(\underbrace{0,\dots,0}_{n-i},\underbrace{q^{e+i-1},\dots,q^{e+i-1}}_{i}, \underbrace{-1,\dots,-1}_{i},\underbrace{0,\dots,0}_{n-i})^\top.
	\]
	Then $v_i$ is an eigenvector of $Q$ with eigenvalue $-q^{(n-1)(n-1+e)}$.
\end{lemma}
\begin{proof}
	Let $w_1,\dots,w_{2n}$ be the row vectors of $Q$. We have to show $w_kv_i=-q^{(n-1)(n-1+e)} (v_i)_k$ for all $1 \leq i,k \leq 2n$, where $(v_i)_k$ is the $k$-th entry of $v_i$. We distinguish four cases.
	
	\noindent Case 1. $k\le n-i$. Then $(v_i)_k=0$. Since the first $2n-k$ entries of $w_k$ and the last $n-i$ entries of $v_i$ are zero, we have $w_kv_i=0$.
	
	\noindent Case 2. $n-i<k\le n$. Then $(v_i)_k=q^{e+i-1}$. Since the first $2n-k$ entries of $w_k$ are zero and the last $n-i$ entries of $v_i$ are zero, we have
	\begin{align*}
		w_kv_i&=-\sum_{\ell=2n-k+1}^{n+i}Q_{k\ell}
		=-q^{n(n+e-3)-e}\left(q^{2n-k+e}+\sum_{\ell=2n-k+1}^{n+i}(q-1)q^{\ell+e-1}\right)
		=-q^{n(n+e-2)+i}.
	\end{align*}
	
	\noindent Case 3. $n<k\le n+i$ (that is $n-i\le 2n-k$). Then $(v_i)_k=-1$ and
	\begin{align*}
		w_kv_i&=\sum_{\ell=2n-k+1}^{n}q^{e+i-1}Q_{k\ell}-\sum_{\ell=n+1}^{n+i}Q_{k\ell}
		\\&=
		q^{n(n+e-3)+i-1}\left(q^{2n+1-k}+\sum_{\ell=2n-k+1}^{n}(q-1)q^\ell\right)
		\\&\ \ \ -
		q^{n(n+e-3)-e}\left(q^{n}(q^e-q)+\sum_{\ell=n+1}^{n+i}(q-1)q^{\ell-1+e}\right)
		\\&=
		q^{n(n+e-3)+i-1}q^{n+1}
		-
		q^{n(n+e-3)-e}(-q^{n+1}+q^{n+e+i})
		=
		q^{n(n+e-2)-e+1}.
	\end{align*}
	
	\noindent Case 4. $n+i<k$ (and hence $n-i+1 > 2n+1-k$). Then $(v_i)_k=0$ and
	\begin{align*}
		w_kv_i&=\sum_{\ell=n-i+1}^{n}q^{e+i-1}Q_{k\ell}-\sum_{\ell=n+1}^{n+i}Q_{k\ell}
		\\&=
		q^{n(n+e-3)+i-1}\sum_{\ell=n-i+1}^{n}(q-1)q^\ell
		-
		q^{n(n+e-3)}\sum_{\ell=n+1}^{n+i}(q-1)q^{\ell-1})
		=0.
	\end{align*}
	In all cases we see that $w_kv_i=-q^{(n-1)(n-1+e)}(v_i)_k$.
\end{proof}

\begin{defin}\label{def:typeBmatrixofeigenvectors}
	For each $j \in [n]$ we can define the $|\cC| \times |\cP|$ matrix $F_j$ as
	\[F_j = q^{j+e-1}\sum_{i=n-j+1}^{n} T_i-\sum_{i=n+1}^{n+j} T_i,\]
	where $T_i$ denotes the type $i$ adjacency matrix.
\end{defin}

\noindent \textbf{Step 3.} We will now prove that the columns of $F_j$, $j \in [n]$, span the part of the $\lambda_{\min}$-eigenspace of $A_{w_0}$ contained in $M_{([n-1],[1])}$.

\begin{lemma}\label{lem:typeBbetas} 
	For all $i,j \in [n]$ we have
	\[|\cC_i^X \cap \cC_j^Y| = \begin{cases}
		\delta_{ij}c(n-1,e)q^{i-1} & \text{if } X = Y, \\
		c(n-2,e)q^{i-2}(q^{j-1}-\delta_{ij}) & \text{if $X$ and $Y$ collinear}, \\
		0 & \text{if $X$ and $Y$ opposite.}
	\end{cases}\]
	For all $i \in [n]$ and $j \in \{n+1,\dots,2n\}$ we have
	\[|\cC_i^X \cap \cC_j^Y| = \begin{cases}
		0 & \text{if } X = Y, \\
		c(n-2,e)q^{i+j+e-4} & \text{if $X$ and $Y$ collinear and $i+j < 2n+1$}, \\
		0 & \text{if $X$ and $Y$ collinear and $i+j = 2n+1$}, \\
		c(n-2,e)q^{i+j+e-5} & \text{if $X$ and $Y$ collinear and $i+j > 2n+1$}, \\
		0 & \text{if $X$ and $Y$ opposite and $i+j < 2n+1$.} \\
		c(n-1,e) & \text{if $X$ and $Y$ opposite and $i+j = 2n+1$.} \\
		c(n-1,e)(q-1)q^{i+j-2n-2} & \text{if $X$ and $Y$ opposite and $i+j > 2n+1$.}
	\end{cases}\]
\end{lemma}
\begin{proof}
	Let $i \in [n]$ and assume that $j \in [n]$. If $c = (U_1,\dots,U_n) \in \cC_i^X \cap \cC_j^Y$ then $X,Y \in U_n$ as $i,j \leq n$. This implies that either $X = Y$ or $X$ and $Y$ are collinear. When $X = Y$, we immediately see that $|\cC_i^X \cap \cC_j^Y| = \delta_{ij}|\cC_i^X|$ as a flag has only one type and we can use \Cref{lem:typeBsizeofCiX}. So fix $X \in \cP$ and count the pairs $(Y,c)$ such that $X$ and $Y$ are collinear and $c \in \cC_i^X \cap \cC_j^Y$ in two ways. There are $|\cC_i^X|=c(n-1,e)q^{i-1}$ possibilities for $c$ and for every such $c$ there are $q^{j-1}-\delta_{ij}$ ways to pick $Y$, as every possibility will give a point collinear to $X$. The $\delta_{ij}$ appears since $Y$ has to be distinct from $X$. On the other hand, for every of the $qv(n-1,e)$ points collinear with $Y$, there are equally many possibilities for $c$ by symmetry, namely $|\cC_i^X \cap \cC_j^X|$. It follows that $c(n-1,e)q^{i-1} \cdot (q^{j-1}-\delta_{ij}) = qv(n-1,e) \cdot |\cC_i^X \cap \cC_j^X|$, from which the statement follows. \\
	
	For the second part of the lemma we assume $i \in [n]$ and $j \in \{n+1,\dots,2n\}$. Now fix $Y \in \cP$ and count the pairs $(X,c)$ such that $X$ and $Y$ are collinear and $c \in \cC_i^X \cap \cC_j^Y$ in two ways. There are $|\cC_j^Y|=c(n-1,e)q^{j+e-2}$ possibilities for $c = (U_1,\dots,U_n)$ by \Cref{lem:typeBsizeofCiX}. The number of possibilities for $X$, given $c$, equals $|(U_i \setminus U_{i-1}) \cap Y^\perp|$. Observe that either $U_i \subseteq Y^\perp$ or $U_i \cap Y^\perp$ is a hyperplane in $U_i$.
	
	If $i < 2n-j+1$, then $U_i \subseteq U_{2n-j} \subseteq Y^\perp$ so that there are $q^{i-1}$ possibilities for $X$. If $i = 2n-j+1$ then $U_{i-1} = U_{2n-j} \subseteq Y^\perp$ and $U_i = U_{2n-j+1} \not\subseteq Y^\perp$ so that $U_i \cap Y^\perp = U_{i-1}$ and hence there are $0$ possibilities for $X$. If $i > 2n-j+1$ then $U_i \cap Y^\perp$ and $U_{i-1}$ are distinct hyperplanes of $U_i$ so that there are $q^{i-2}$ possibilities for $X$.
	
	On the other hand, for every of the $qv(n-1,e)$ points collinear with $Y$, there are equally many possibilities for $c$ by symmetry, namely $|\cC_i^X \cap \cC_j^X|$. We conclude that 
	\[|\cC_i^X \cap \cC_j^X| \cdot qv(n-1,e) = c(n-1,e)q^{j+e-2}\cdot\begin{cases}
		q^{i-1} & \text{if } i+j < 2n+1, \\
		0 & \text{if } i+j = 2n+1, \\
		q^{i-2} & \text{if } i+j > 2n+1.
	\end{cases}\]
	
	Finally, we fix $Y \in \cP$ and count the pairs $(X,c)$ such that $X$ and $Y$ are opposite and $c \in \cC_i^X \cap \cC_j^Y$ in two ways. It still holds that there are $|\cC_j^Y|=c(n-1,e)q^{j+e-2}$ possibilities for $c = (U_1,\dots,U_n)$. The number of possibilities for $X$, given $c$, now equals $|U_i \setminus (U_{i-1} \cup Y^\perp)|$. Observe that either $U_i \subseteq Y^\perp$ or $U_i \cap Y^\perp$ is a hyperplane in $U_i$.
	We can recycle the arguments from the previous paragraph to see that the number of possibilities for $X$ is $0$ if $i < 2n-j+1$, $q^{i-1}$ if $i=2n-j+1$, and $(q-1)q^{i-2}$ if $i > 2n-j+1$.
	
	On the other hand, for every of the $q^{2n+e-2}$ points opposite to $Y$, there are equally many possibilities for $c$ by symmetry, namely $|\cC_i^X \cap \cC_j^X|$. We conclude that 
	\[|\cC_i^X \cap \cC_j^X| \cdot q^{2n+e-2} = c(n-1,e)q^{j+e-2}\cdot\begin{cases}
		0 & \text{if } i+j < 2n+1, \\
		q^{i-1} & \text{if } i+j = 2n+1, \\
		(q-1)q^{i-2} & \text{if } i+j > 2n+1.
	\end{cases}\]
\end{proof}

In type $B_n$, we have a two-class association scheme on $\cP$: the basis matrices $\{A_0,A_1,A_2\}$ are $A_0 = I_\cP$ and $A_1$ and $A_2$ are the adjacency matrices of the collinearity and opposition graphs on points respectively. 
In this case, the $P$-matrix is 
\begin{align}\label{eq:typeBPmatrix}
	\begin{pmatrix}
	1 & qv(n-1,e) & q^{2n+e-2} \\
	1 & q^{n-1}-1 & -q^{n-1} \\
	1 & -q^{n+e-2}-1& q^{n+e-2}
\end{pmatrix}
\end{align}
see for example \cite[Theorem 2.2.12]{BVM22} and $E_s = E_1$ is the special idempotent corresponding to the module indexed by $([n-1],[1])$. With this in hand, we are now in the position to state the main theorem of this section.

\begin{theorem}
	Define $g:[n]\rightarrow[n]$ by $g(j):=n-j+1$. Then $g$ satisfies the triangular criterion.
\end{theorem}

\begin{proof}
	We will do so by checking the conditions in \Cref{lem:checktrianglecriterion} for all $j \in [n]$, which are
	\begin{align*}
		&\sum_{i = 1}^{2n} f_{ij} t_{hi}^k = 0 \text{ for all } k \in \{0,1,2\} \text{ and } h < n-j+1, \text{ and } \\
		&\sum_{i = 1}^{2n} \sum_{k=0}^2 f_{ij} t_{n-j+1,i}^kp_k(r) \begin{cases}
			= 0 & \text{if } r \in \{0,2\}, \\
			\neq 0 & \text{if } r = 1.
		\end{cases}
	\end{align*}	
	
	Recall by \Cref{def:typeBmatrixofeigenvectors} that for all $j \in [n]$ we have
		\[f_{ij} = \begin{cases}
			q^{j+e-1} & \text{if } n-j+1 \leq i \leq n, \\
			-1 & \text{if } n+1 \leq i \leq n+j, \\
			0 & \text{otherwise.}
		\end{cases}\]	
	
	Looking at \Cref{lem:typeBbetas}, the indicator $\delta_{hi}$ appearing in some $t^k_{hi}$ will never be zero as $h < n-j+1$,  and moreover $h+i < 2n+1$ if $i \leq n+j$. The first condition is therefore trivially true for $k = 0$ and $k = 2$. For $k = 1$ it simplifies to
	\[c(n-2,e)\left(\sum_{i = n-j+1}^n q^{h+i+j+e-4} - \sum_{i = n+1}^{n+j} q^{h+i+e-4}\right) = 0,\]
	which is indeed true.
	
	On the other hand, when $h = g(j) = n-j+1 \leq n$, we will deduce the coefficients of $A_0$, $A_1$ and $A_2$ in the expansion of $T_h^\top F_j$ (see \eqref{eq:productTFinassocscheme}) as an intermediate step. The coefficient of $A_k$ is $\sum_{i = 1}^{2n} f_{ij} t_{n-j+1,i}^k$ which can be computed in a similar way as before. The main differences are that now the terms involving $\delta_{hi}$ matter as $h = i = n-j+1$ occurs and the last term ($i = n+j$) satisfies $h+i = 2n+1$. Making the correct adjustments, we find the coefficients
	\begin{align*}
		\begin{cases}
			c(n-1,e)q^{n+e-1} & k = 0 \\
			c(n-2,e)q^{n+e-2}(q^{n-1}-1)& k = 1 \\
			-c(n-1,e) & k = 2.
		\end{cases}
	\end{align*}

	Denote these coefficients by $\alpha_k$, $k \in \{0,1,2\}$. A tedious but straightforward computation using the $P$-matrix \eqref{eq:typeBPmatrix} now verifies that $\sum_{k=0}^2 \alpha_k p_k(r) = 0$ for $r \in \{0,2\}$ and equals \[c(n-2)\left(v(n-1,e)(q^{n+e-1}+q^{n-1})+q^{n+e-2}(q^{n-1}-1)^2\right) \neq 0\] for $r = 1$.
\end{proof}

With an identical proof as in \Cref{cor:typeAmainresult}, we can thus conclude the following result. Recall from \Cref{maintheoremprev} that $-q^{(n-1)(n+e-1)}$ equals $\lambda_{\min}$ unless $n$ is odd and $e \geq 1/2$.

\begin{cor}\label{cor:typeBmainresult}
	The $(-q^{(n-1)(n+e-1)})$-eigenspace for $A_{w_0}$ in type $B_n$, $n \geq 2$, contained in the module indexed by $([n-1],[1])$ is $\col(F_1) \oplus \dots \oplus \col(F_n)$.
\end{cor}

\subsection{Type $B_n$: module $M_{(\emptyset,[n])}$} \label{sec:typeBhardmodule}

This module is more special since it already appears as an irreducible module in $\cA(G,G/P_J)$, where $J = [n-1]$ and hence $P_J$ is the parabolic subgroup stabilizing of a generator \cite[Section 3.2.3]{DBMM22}. Moreover ${\rm mult}_{(\emptyset,[n])}=1$ and hence eigenvectors in this module can be found by lifting eigenvectors for the adjacency matrix of the opposition graph on generators. More precisely, as seen in \cite[Remark 2.19]{DBMM22}, we obtain a quotient matrix of $A_{w_0}$ by partitioning the maximal flags according to the generator they contain. The quotient matrix obtained in this way is a scalar multiple of the adjacency matrix of the opposition graph on generators. Eigenvectors for the smallest eigenvalue in this module of the latter graph have been determined by Eisfeld \cite{Eisfeld99}, but their description is rather cumbersome. \\

Alternatively, one could try our previous approach with $J = [n-1]$. Now we have $|W:W_J|=2^n$ orbits, which follows from the fact that $W$ is the hyperoctahedral group of order $2^n\cdot n!$ and $W_J = \Sym(n)$. This implies that the quotient matrices under investigation are still of exponential size in $n$, which makes computations by hand rather complicated. However, this partition can be made coarser by fixing a generator and looking at the mutual position of this generator and the generator in a given maximal flag. Now there are only $n+1$ possibilities, and one could try to continue our previous approach. Nevertheless, computations for small ranks suggest that finding a reasonable description of these eigenvectors seems unlikely which is not too surprising given Eisfeld's results. Given that these eigenvectors will most likely not yield useful information towards classifying maximal EKR-sets anyway, we will not pursue this direction any further.

\subsection{Type $D_n$}

Geometrically, the oriflamme complex is a minor modification of $\PS(n,0,q)$ \cite[Section 4.5.2]{BVM22}. The generators in this polar space fall into two classes, see \cite[Theorem 2.2.17]{BVM22}, such that every rank $n-1$ space in $\PS(n,0,q)$ is contained in exactly one generator of each class. These classes are sometimes referred to as the Greeks and Latins, but we will simply label them with $+$ and $-$. 

This extra information is recorded in the oriflamme geometry as follows. Whereas a maximal flag in the polar space would be of the form $(U_1,\dots,U_{n-1},U_n)$, we will now keep track of the two generators containing $U_{n-1}$ by defining a flag to be of the form $(U_1,\dots,U_{n-2},U_n^-,U_n^+)$, where $U_n^\varepsilon$ is the rank $n$ space containing $U_{n-1}$ of type $\varepsilon \in \{+,-\}$. Two generators from distinct classes are incident whenever they intersect in a rank $n-1$ space of the polar space. In other words, we forget about the rank $n-1$ spaces as elements of the geometry, but they are implicitly present as the intersections $U_n^- \cap U_n^+$ of incident generators. It is hence no surprise that we will be able to recycle many of the results, and we can be rather brief in this section.  \\

Recall that the bipartite double, also called the double cover, of a graph $G$ with adjacency matrix $A$ is the bipartite graph with adjacency matrix $$A' = \begin{pmatrix} 0 & A \\ A & 0	\end{pmatrix}.$$ If $v$ is an eigenvector of $A$ with eigenvalue $\lambda$, then $(v,v)^\top$ and $(v,-v)^\top$ are eigenvectors of $A'$ with eigenvalues $\lambda$ and $-\lambda$ respectively. Vice versa, if $(v,w)^\top$ is an eigenvector of $A'$ with eigenvalue $\lambda$  and $v \neq -w$, then $v+w$ is an eigenvector for $A$ with the same eigenvalue.

\begin{lemma}\label{lem:typeDfromtypeB}
	The opposition graphs in type $B_n$, $e = 0$, and type $D_n$ are related in the following ways. When $n$ is even, the former consists of two disjoint copies of the latter. When $n$ is odd, the former is the bipartite double of the latter.
\end{lemma}
\begin{proof}
	For brevity, we will denote $c = (U_1,\dots,U_{n-2},U_n^{-},U_n^{+})$ and $d = (V_1,\dots,V_{n-2},V_n^{-},V_n^{+})$ for two flags in type $D_n$ and their representation in type $B_n$ as $c^\varepsilon = (U_1,\dots,U_{n-1},U_n^{\varepsilon})$, where $\varepsilon \in \{-,+\}$ and $U_{n-1} = U_n^{-}\cap U_n^{+}$, and similarly denote $d^\varepsilon$.
	
	It is known, see for example \cite[Corollary 3.27]{DBMM22}, that two generators in a hyperbolic quadric of the same class can be opposite if and only if $n$ is even. It follows that the opposition graph in type $B_n$, $e = 0$, $n$ even is bipartite. Furthermore, two flags $c^-$ and $d^-$ in one component are adjacent if and only if $c^+$ and $d^+$ in the other component are. This follows from the fact that $\dim(U_n^{+} \cap V_n^{+}) \equiv \dim(U_n^{-} \cap V_n^{-}) \equiv n \pmod 2$ and $\dim(U_{n-1} \cap V_{n-1}) = 0$. Since $c$ and $c$ are opposite if and only if both pairs $(c^-,d^-)$ and $(c^+,d^+)$ are, we conclude that the map $\phi_\varepsilon:(U_1,\dots,U_{n-1},U_n^{-},U_n^{+}) \mapsto (U_1,\dots,U_{n-1},U_n^{\varepsilon})$ defines a graph isomorphism between the opposition graphs in type $D_n$ and $B_n$ for $\varepsilon \in \{-,+\}$.
	
	When $n$ is odd, we deduce similarly that $c^-$ and $d^+$ are adjacent in the opposition graph in type $B_n$ if and only if $c^+$ and $d^-$ are if and only if $c$ and $d$ are adjacent in the opposition graph in type $D_n$.
\end{proof}

In both cases, a spanning set of eigenvectors follows quickly from the results in type $B_n$. Denote by $\cC$ and $\cP$ the set of maximal flags and points in a geometry of type $D_n$ respectively. 

\begin{defin}\label{def:typeDmatrixofeigenvectors}
	For each $j \in [n]$ we can define the $|\cC| \times |\cP|$ matrix $F_j'$ as
	\[F'_j(c,X) = F_j(c^+,X) + F_j(c^-,X) \text{ for all } c \in \cC, X \in \cP,\]
	where $F_j$ was defined in \Cref{def:typeBmatrixofeigenvectors} and $c^+,c^-$ are the representations to maximal flags in type $B_n$ as before.
\end{defin}

Observe that $F_j(c^+,X) = - F_j(c^-,X)$ for fixed $X \in \cP$ could happen for $j = 1$ and some $c \in \cC$, but not all of them. Combined with \Cref{lem:typeDfromtypeB} follows that we indeed find non-zero vectors which are eigenvectors for $A_{w_0}$ in type $D_n$.

\begin{cor}\label{cor:typeDmainresult}
		The $\lambda_{\min}$-eigenspace for $A_{w_0}$ in type $D_n$, $n \geq 4$. is $\col(F'_1) \oplus \dots \oplus \col(F'_n)$.
\end{cor}

\newpage

\appendix

\section{Explicit values for the multiplicity of the smallest eigenvalue of the opposition graph}

In this table we list the explicit values for the multiplicity of $\lambda_{\min}$ for each case. We will use the Cartan notation as in \Cref{table:classicalgroups}, so that the value of $e$ in type $B_n$ is implicit. Remark that the polar spaces with Cartan notation $B_n(q)$ and $C_n(q)$ have isomorphic Iwahori-Hecke algebras and hence also the same multiplicities of eigenvalues of $A_{w_0}$. All values follow from \Cref{thm:multonmodules}.

	\begin{table}[h!]
	\setlength{\tabcolsep}{3mm}
	\def\arraystretch{2.5} 
	\begin{center}
		\begin{tabular}{c | c || c  | c}
			case & multiplicity & case &  multiplicity \\ \hline
			$A_{2n}(q)$ & $n\dfrac{q^{2n+1}-q}{q-1}$ & 
			$A_{2n-1}(q)$ & $n\dfrac{q^{2n}-q}{q-1}$ \\
			$^2D_{n+1}(q^2)$ & $n\dfrac{q^2(q^{2n}-1)}{q^2-1}$ & 
			$^2A_{2n}(q^2)$ & $n\dfrac{q^3(q^{2n}-1)(q^{2n-1}+1)}{(q^2-1)(q+1)}$ \\
			
			$^2A_{4n-1}(q^2)$ & $2n\dfrac{q^2(q^{4n}-1)(q^{4n-3}+1)}{(q^2-1)(q+1)}$ & $^2A_{4n-3}(q^2)$ & $\dfrac{q(q^{4n-3}+1)}{q+1}$ \\
			$B_{2n}(q)$ & $2n\dfrac{q(q^{2n}-1)(q^{2n-1}+1)}{q-1}$ & 
			$D_{2n}(q)$ & $2n\dfrac{(q^{2n+1}-q)(q^{2n-2}+1)}{(q^2-1)}$ \\
			$B_{2n-1}(q)$ & \multicolumn{3}{l}{\hspace{.1cm}$(2n-1)\dfrac{q(q^{2n-1}-1)(q^{2n-2}+1)}{2(q-1)}+\dfrac{q(q^{2n-1}+1)(q^{2n-2}+1)}{2(q+1)}$}  \\
			
		\end{tabular}
	\caption{The multiplicities of the smallest eigenvalue of $A_{w_0}$}
	\label{table:mults}
	\end{center}
\end{table}

%

\end{document}